\documentclass[11pt]{aims}
\usepackage{amsmath, amssymb, mathrsfs}
  \usepackage{paralist}
  \usepackage{graphics} %% add this and next lines if pictures should be in esp format
  \usepackage{epsfig} %For pictures: screened artwork should be set up with an 85 or 100 line screen
\usepackage{graphicx}
  \usepackage{epstopdf}%This is to transfer .eps figure to .pdf figure; please compile your paper using PDFLeTex or PDFTeXify.
 \usepackage[colorlinks=true]{hyperref}
 \hypersetup{urlcolor=blue, citecolor=red}
 \usepackage[toc,page]{appendix}
\usepackage{chngcntr}

\usepackage{titlesec}
\titleformat{\section}{\Large\bfseries}{\thesection.}{4pt}{}
\titleformat{\subsection}{\large\bfseries}{\thesection.\arabic{subsection}.}{4pt}{}
\titleformat{\subsubsection}{\bfseries}{\thesection.\arabic{subsection}.\arabic{subsubsection}.}{4pt}{}
\titleformat*{\paragraph}{\bfseries}
\titleformat*{\subparagraph}{\bfseries}
 \usepackage[margin=1.1in]{geometry}

  \def\currentyear{}
   \def\currentmonth{}
    \def\ppages{}
     \def\DOI{}

\newtheorem{theorem}{Theorem}[section]

\newtheorem{lemma}[theorem]{Lemma}
\newtheorem{proposition}[theorem]{Proposition}
\theoremstyle{definition}
\newtheorem{definition}[theorem]{Definition}
\newtheorem{remark}[theorem]{Remark}

\newcommand{\RN}{\mathbb{R}^N}
\newcommand{\Rb}{\mathbb{R}}

\newcommand{\Cc}{\mathcal{C}}
\newcommand{\Dc}{\mathcal{D}}

\newcommand{\Sc}{\mathcal{S}}

\newcommand{\Uc}{\mathcal{U}}
\newcommand{\Vc}{\mathcal{V}}

\newcommand{\Ls}{\mathscr{L}}

\numberwithin{equation}{section}

\title[Blowup for a critical nonlinear heat equation] %use the shortened version of the full title
      {Blowup solutions for a nonlinear heat equation involving a critical power nonlinear gradient term}

\author[T. Ghoul, V. T. Nguyen, H. Zaag]{}

\thanks{H. Zaag is supported by the ERC Advanced Grant no. 291214, BLOWDISOL and by the ANR project ANA\'E ref. ANR-13-BS01-0010-03. \\ -----------------\\ \today}
\begin{document}
\maketitle

% Enter the first author's name and address:
\centerline{\scshape Tej-Eddine Ghoul$^\dagger$, Van Tien Nguyen$^\dagger$ and Hatem Zaag$^\ast$}
\medskip
{\footnotesize
 \centerline{$^\dagger$New York University in Abu Dhabi, P.O. Box 129188, Abu Dhabi, United Arab Emirates.}
  \centerline{$^\ast$Universit\'e Paris 13, Sorbonne Paris Cit\'e, LAGA, CNRS (UMR 7539), F-93430, Villetaneuse, France.}
}

%% Enter the first author's name and address:
%\centerline{\scshape Tej-Eddine Ghoul and Van Tien Nguyen}
%\medskip
%{\footnotesize
% \centerline{ New York University Abu Dhabi, Department of Mathematics,}
%   \centerline{Saadiyat Island, P.O. Box 129188, Abu Dhabi, United Arab Emirates.}
%}
%
%
%\medskip
%
%\centerline{\scshape Hatem Zaag}
%\medskip
%{\footnotesize
% \centerline{Universit\'e Paris 13, Sorbonne Paris Cit\'e}
%   \centerline{LAGA, CNRS (UMR 7539), F-93430, Villetaneuse, France.}
%}

\bigskip

\begin{abstract} We consider the following exponential reaction-diffusion equation involving a nonlinear gradient term:
$$\partial_t U = \Delta U + \alpha|\nabla U|^2 + e^U,\quad (x, t)\in\RN\times[0,T), \quad \alpha > -1.$$
We construct for this equation a solution which blows up in finite time $T > 0$ and satisfies some prescribed asymptotic behavior. We also show that the constructed solution and its gradient blow up in finite time $T$ simultaneously at the origin, and find precisely a description of its final blowup profile. It happens that the quadratic gradient term is critical in some senses, resulting in the change of the final blowup profile in comparison with the case $\alpha = 0$. The proof of the construction inspired by the method of Merle and Zaag in 1997, relies on the reduction of the problem to a finite dimensional one, and uses the index theory to conclude. One of the major difficulties arising in the proof is that outside the \textit{blowup region}, the spectrum of the linearized operator around the profile can never be made negative. Truly new ideas are needed to achieve the control of the outer part of the solution.  Thanks to a geometrical interpretation of the parameters of the finite dimensional problem in terms of the blowup time and the blowup point, we obtain the stability of the constructed solution with respect to perturbations of the initial data.
\end{abstract}

\section{Introduction.}
We are interested in the following nonlinear heat equation: 
\begin{equation}\label{equ:problem}
\left\{
\begin{array}{rcl}
\partial_t U &=& \Delta U + \alpha |\nabla U|^r + e^U,\\
U(0) &=& U_0,
\end{array}
\right.
\end{equation}
where $U(t): x \in \mathbb{R}^N \to U(x,t) \in \mathbb{R}$, $\Delta$ and $\nabla$ stand for the Laplacian and the gradient in $\mathbb{R}^N$ with $N \geq 1$,
$$r = 2 \quad \text{and}\quad \alpha > -1.$$ 
Equation \eqref{equ:problem} can be viewed as the limiting case of the following critical splitting as $p \to +\infty$, which was introduced by Chipot and Weissler \cite{CWjma89}:
\begin{equation}\label{equ:sheGra}
\partial_t U = \Delta U + \alpha |\nabla U|^r + |U|^{p-1}U, \quad \text{with}\;\; p > 1\; \text{and}\;r = \frac{2p}{p + 1}.
\end{equation}

The Cauchy problem for \eqref{equ:problem} can be solved in several functional spaces $\mathcal{F}$, for example $\mathcal{F} = W^{1,\infty}(\RN)$ or in a special affine space $\mathcal{F} = \mathcal{H}_{a}$ for some positive constant $a$ with
\begin{equation}\label{def:Ha1a2}
\mathcal{H}_{a} = \{u \in \psi + W^{1, \infty}(\RN) \; \text{with}\; \psi(x) = -\ln(1 + a|x|^2)\}.
\end{equation}
In particular, the problem \eqref{equ:problem} has  a unique classical solution $U(t) \in \mathcal{F}$ defined on $[0, T)$ with $T \leq +\infty$ (see Remark \ref{rem:Cauchypro} for more details). In the case $T = +\infty$, we have the global existence for $U(t)$; on the contrary, i.e $T < +\infty$, we say that the solution $U(t)$ blows up in finite time $T$, namely
\begin{align*}
&\lim_{t \to T}\|U(t)\|_{W^{1, \infty}(\RN)} = +\infty,\\
\text{or}\;\;&\lim_{t \to T}\|U(t) - \psi\|_{W^{1, \infty}(\RN)} = +\infty \;\;\text{in the second case.}
\end{align*}

As for the case of equation \eqref{equ:problem}, the value $r = 2$ is a critical exponent for different reasons ($r < 2$ and $r > 2$ correspond to the sub-critical and super-critical cases). One reason is that, when $r = 2$, equation \eqref{equ:problem} is invariant under the following transformation:
\begin{equation}\label{equ:invp1}
\forall \lambda > 0, \quad U_\lambda(x, t) = 2\ln\lambda + U(\lambda x, \lambda^2 t),
\end{equation}
as for the equation without the gradient term, i.e $\alpha = 0$. Recalling that equation \eqref{equ:sheGra} is invariant under the transformation
\begin{equation}\label{equ:invp2}
\forall \lambda > 0, \quad U_\lambda(x, t) = \lambda^\frac{2}{p-1} U(\lambda x, \lambda^2 t).
\end{equation}

Let us now sketch the main results for the case of the equation
\begin{equation}\label{eq:expU}
\partial_t U = \Delta U + e^U.
\end{equation}
One of the first result for the finite-time blow up problem \eqref{eq:expU} is due to Friedman and McLoed \cite{FM85iumj} who proved the following upper and lower bounds for the blowup rate of $U$ under some conditions on the initial data,
\begin{equation}\label{equ:blrate}
-c \leq U(0,t) + \ln(T-t)\leq C,
\end{equation}
for some $c, C$ positive (see also Berbenes and Eberly \cite{BEbook89} for this estimate).

The study of the blowup behavior for solution \eqref{eq:expU} is done through the introduction of similarity variables
\begin{equation}\label{def:changeVar}
W_a(y,s) = U(x,t) + \ln(T-t), \quad y = \frac{x - a}{\sqrt{T-t}}, \quad s = -\ln (T-t),
\end{equation}
where $a$ may or not be a blowup point for $U$. From \eqref{eq:expU}, we see that $W_a$ satisfies the following equation:  for all $(y,s) \in \RN \times [-\ln T, +\infty)$,
\begin{equation*}
\partial_sW_a = \Delta W_a - \frac{y}{2}\cdot \nabla W_a + e^{W_a} - 1.
\end{equation*}
Since $s \to +\infty$ as $t \to T$, the change of variables \eqref{def:changeVar} converts any question about the blowup of $U$ into one about the large time asymptotic of $W_a$.

According to Berbenes and Eberly \cite{BEbook89}, we know that if $U(x,t)$ is a solution of \eqref{eq:expU} which blows up at $x = a$ and $t = T$, then 
\begin{equation}\label{equ:bllimit}
\lim_{t \to T} \big[U(a + y\sqrt{T-t},t) + \ln(T-t)\big] = \lim_{s \to +\infty} W_a(y,s) = 0,
\end{equation}
uniformly on compact sets $|y| \leq R$. 

This estimate has been refined until the higher order by Bebernes and Briche \cite{BBsima92}, Herrero and Vel\'azquez \cite{HVaihn93} in one dimensional case. More precisely, they classified the behavior of $W_a$ for $|y|$ bounded, and showed that one of the following cases occurs:
\begin{equation}\label{eq:case1class}
\sup_{|y| \leq R} \left| W_a(y,s) + \frac{1}{4s}(y^2 - 2)\right| =  o\left(\frac{1}{s}\right),
\end{equation}
or
\begin{equation*}
\sup_{|y| \leq R}|W_a(y,s)| = \mathcal{O}(e^{-\mu s}) \quad \text{for some $\mu > 0$,}
\end{equation*}
(the exponential convergence has been refined up to order $1$ in \cite{HVaihn93}). It is remarkable that a similar result can be extended to higher dimensional cases by using the technique of \cite{VELcpde92} for the equation
\begin{equation}\label{eq:Up}
\partial_t U = \Delta U + |U|^{p-1}U.
\end{equation}

When \eqref{eq:case1class} occurs, the authors of \cite{BBsima92} established the following blowup profile in the variable $z = \frac{y}{\sqrt{s}}$ (which is the intermediate scale that separates the regular and singular parts)
\begin{equation}\label{equ:blprofile}
\lim_{t \to T} \big[U(a + z \sqrt{(T-t)|\ln(T-t)|}, t) + \ln(T-t)\big] = \lim_{s \to +\infty}W_a(z\sqrt{s},s) = \Phi_0(z),
\end{equation}
uniformly on compacts in $z$, where 
\begin{equation}\label{def:Phi0xi}
\Phi_0(z)=  -\ln\left( 1 + \frac{|z|^2}{4}\right).
\end{equation}
Note that \eqref{equ:blprofile} was formally obtained in \cite{Doldsjam89}, \cite{Doldqjmam85} and \cite{GHVsovi91} by means of the method of matched asymptotic expansions. The asymptotic behavior \eqref{equ:blprofile} leads to the limiting profile in the $U(x,t)$ variable, in the sense that $U(x,t) \to U^*_0(x)$ when $t \to T$ if $x \ne a$ and $x$ is in the neighborhood of $a$ with
\begin{equation}\label{equ:finalprofile}
U^*_0(x) \sim -2\ln |x - a| + \ln|\ln|x - a|| +\ln 8.
\end{equation}
In \cite{Breiumj90} and \cite{Brejde92}, Bressan proved the existence of solutions to \eqref{eq:expU} which blow up in finite time and verify the behavior \eqref{equ:blprofile}. He also obtained the stability of such blowup behavior with respect to small perturbations of the the initial data. For more results related to equation \eqref{eq:expU}, see Fila and Pulkkinen \cite{FPtmj08}, \cite{PULmmas11} and the references therein.

\bigskip

By considering $\alpha \neq 0$ and $r = 2$ in \eqref{equ:problem}, we want to ask whether the nonlinear gradient term appearing in the equation affects the blowup profile of the case $\alpha = 0$, i.e equation \eqref{eq:expU}. Note that the case $r < 2$ can be considered as a perturbation of equation \eqref{eq:expU} because the nonlinear gradient term has a sub-critical size in the sense that in the similarity variables setting \eqref{def:changeVar}, equation \eqref{equ:problem} yields
\begin{equation}\label{eq:perSHE}
\partial_sW_a = \Delta W_a - \frac{y}{2}\cdot \nabla W_a + \alpha e^{-\delta s} |\nabla W_a|^r + e^{W_a} - 1, \quad\text{with}\;\; \delta = 1 - \frac{r}{2}.
\end{equation}
This gives another explanation to the fact that problem \eqref{equ:problem} is critical when $r = 2$. In fact, our problem is motivated by the work of Tayachi and Zaag \cite{TZpre15} treated for equation \eqref{equ:sheGra}. In that paper, the authors construct for \eqref{equ:sheGra} a solution which blows up in finite time $T$ only at the origin and give a sharp description of its blowup profile in the case where $\alpha > 0$ and $p > 3$. The originality of \cite{TZpre15} lays in the fact that the constructed solution does not exist in the case of the standard nonlinear heat equation \eqref{eq:Up}. In particular, their solution has a profile depending on the variable 
\begin{equation}\label{def:varTZ}
z = \frac{x}{|\ln (T-t)|^\mu \sqrt{T-t}} \quad \text{with}\quad \mu = \frac{p+1}{2(p-1)}> \frac{1}{2},
\end{equation}
which is different from the results for equation \eqref{eq:Up} by Herrero and Vel\'azquez \cite{HVaihn93}, \cite{VELcpde92}, \cite{VELtams93}, Bricmont and Kupianen \cite{BKnon94}, where the blowup profiles depend on the reduced variables
$$z = \frac{x}{|\ln(T-t)|^{\frac 12}\sqrt{T-t}} \quad \text{or} \quad z = \frac{x}{(T-t)^{\frac 1{2m}}} \;\; \text{where $m \geq  2$ is an integer.}$$
This evidently shows the effect of the forcing gradient term in equation \eqref{equ:sheGra} with $\alpha > 0$ in the equation. It is worth to mention the work by Ebde and Zaag \cite{EZsema11} where the authors considered equation \eqref{equ:sheGra} in the case where $r$ is sub-critical, i.e $r < \frac{2p}{p+1}$. They showed that the involved nonlinear gradient term in \eqref{equ:sheGra} does not affect to the final blowup profile, leading to the same result (construction and stability of a solution whose blowup profile depends on the variable $z = \frac{x}{|\ln (T-t)|^\frac 12 \sqrt{T-t}}$) for the standard nonlinear heat equation \eqref{eq:Up} obtained in \cite{BKnon94} and \cite{MZdm97}. For more results related to \eqref{equ:sheGra}, see Galaktionov and V\'azquez \cite{GVsiam93}, \cite{GVjde96}, Snoussi, Tayachi and Weissler \cite{SThmj07}, Chipot and Weissler \cite{CWjma89}, Fila \cite{FILpams91}, Souplet, Tayachi and Weissler \cite{STWiumj96}, \cite{SOUjde01}, \cite{STcm01}, Chleb\'ik, Fila and Quittner \cite{CFQdcds03}.\\

In this paper, we aim at constructing a solution for equation \eqref{equ:problem} and giving precisely the description of its blowup profile. Our main result is the following:

\begin{theorem}[\textbf{Existence of a blowup solution for equation \eqref{equ:problem} with the description of its profile}] \label{theo:1} Let $\alpha > -1$ and $r = 2$, there exists $T_0 \in (0,1)$ such that for each $T \in (0, T_0]$, equation \eqref{equ:problem} has a solution $U(x,t)$ such that $U$ and $\nabla U$ blow up in finite time $T$ at the origin simultaneously. Moreover, 
\begin{itemize}
\item[(i)] For all $t \in [0, T)$, 
\begin{equation}
\left\|(T-t)e^{U(x,t)}  - e^{\Phi_{\alpha}\left(z\right)} \right\|_{W^{1,\infty}(\RN)} \leq \frac{C}{\sqrt{|\ln(T-t)|}},\label{equ:behTh1}
\end{equation}
where $C$ is some positive constant and
\begin{equation}\label{def:Phi_alpha}
\Phi_\alpha(z) = -\ln\left(1 + \frac{|z|^2}{4 + 4\alpha}\right), \quad z = \frac{x}{\sqrt{(T-t)|\ln (T-t)|}}.
\end{equation}
\item[(ii)] The functions $e^U$ and $\nabla U$ blow up at the origin and only there.
\item[(iii)] There exists a function $U^*_\alpha(x)$ defined on $\RN \setminus \{0\}$ such that $U(x,t) \to U^*_\alpha(x)$ and $\nabla U(x,t) \to \nabla U^*_\alpha(x)$ as $t \to T$, uniformly on compact sets of $\RN \setminus \{0\}$, where
\begin{equation}\label{equ:finalprofileUlpha}
U^*_\alpha(x) \sim \ln \left(\frac{(8 + 8\alpha)|\ln |x||}{|x|^2}\right) \quad \text{as $|x| \to 0$,}
\end{equation}
and 
\begin{equation}\label{estNalUalpha}
|\nabla U^*_\alpha(x)| \leq \frac{C}{|x|} \quad \text{as $|x| \to 0$.}
\end{equation}
\end{itemize}
\end{theorem}

\begin{remark}Although our problem can be considered as the limiting case of \eqref{equ:sheGra} treated in \cite{TZpre15}, the construction is far from a simple adaptation of the method in \cite{TZpre15}. Compared to the paper \cite{TZpre15}, our work has the following major difficulty: the spectrum of the linearized operator around the profile outside the \textit{blowup region} can never be made negative. This requires new ideas as far as the control of the outer component is concerned, and this is one of the main novelties of our approach. Note that the linearized operator in the power case (equation \eqref{equ:sheGra}) behaves as one with fully negative spectrum, which greatly simplifies the analysis in the outer region.
\end{remark}

\begin{remark} \label{rem:Cauchypro} The initial data for which equation \eqref{equ:problem} has a solution blowing up in finite time $T$ at the origin and verifying the behavior \eqref{equ:behTh1} and \eqref{equ:finalprofileUlpha} is given by  formulation \eqref{def:U0d0d1} below. Note that the initial datum $U_0(x)$ defined in \eqref{def:U0d0d1} does not belong to $\in L^\infty(\RN)$. However, the local in time Cauchy problem for equation \eqref{equ:problem} is well-posed in $\mathcal{H}_{a}$, where $\mathcal{H}_{a}$ is defined by \eqref{def:Ha1a2}. 
Indeed, if we define $\tilde{U}(x,t)$ by $U(x,t) = \tilde{U}(x,t) + \psi(x)$, equation \eqref{equ:problem} is equivalent to 
\begin{equation}\label{equ:tilU}
\partial_t \tilde{U} = \Delta \tilde{U} +  2\alpha \nabla \psi \cdot \nabla \tilde{U} + \alpha |\nabla \tilde{U}|^2 + e^{\psi}e^{\tilde{U}} + \alpha|\nabla\psi|^2 + \Delta \psi.
\end{equation}
Since $\|e^\psi\|_{L^\infty(\RN)}$, $\|\nabla \psi\|_{L^\infty(\RN)}$ and $\|\Delta \psi\|_{L^\infty(\RN)}$ are bounded, we see by classical arguments that equation \eqref{equ:tilU} is well-posed in $W^{1, \infty}(\RN)$. Thus, there exists a unique maximum solution on $[0,T)$ for equation \eqref{equ:problem} with $T \leq +\infty$ corresponding to the initial datum  $U_0$ given by \eqref{def:U0d0d1}.
\end{remark}

\begin{remark} The derivation of the profile $\Phi_\alpha$ can be understood through a formal analysis given in Section \ref{sec:formalapproach} below. We impose the condition $\alpha > -1$ in order to have a bounded profile. 
\end{remark}

\begin{remark} From part $(iii)$ of Theorem \ref{theo:1}, we see that the point $x = 0$ is an isolated blowup point. In order to prove this result, we need to establish a new \textit{"no blowup under some threshold"} of a parabolic inequality with a nonlinear gradient term, see Proposition \ref{prop:noblowup} below.
\end{remark}

\begin{remark} The results stated in Theorem \ref{theo:1} also hold for more general equations than \eqref{equ:problem}, that is 
$$\partial_t U = \Delta U + \alpha|\nabla U|^2 + e^U + F(U,\nabla U),$$
where $F: \Rb \times \RN \to \Rb$ is Lipschitz and satisfies
$$|F(u, v)| \leq C(1 + e^{r_1u} + |v|^{r_2}), \quad r_1 \in [0, 1), \;\; r_2 \in [0, 2).$$
Indeed, the nonlinear term $F(U, \nabla U)$ can be considered as subcritical (in size) in comparison with the main nonlinear terms in the equation, therefore, a simple perturbation of our proof can be extended to this case without difficulties.
\end{remark}

\begin{remark} We see from part $(i)$ of Theorem \ref{theo:1} that the constructed non self-similar solution has the profile depending on the variable 
$$z = \frac{x}{\sqrt{(T-t)|\ln(T-t)|}},$$
which is the same for the case of equations \eqref{equ:problem} and \eqref{equ:sheGra} without the nonlinear gradient term ($\alpha = 0$). Unlike the work done by Tayachi and Zaag in \cite{TZpre15}, where the authors constructed for equation \eqref{equ:sheGra} with $\alpha > 0$ and $ p > 3$ a stable blowup solution having a blowup profile depending on the variable \eqref{def:varTZ} and satisfying 
$$U(x,t) \sim \left(p-1 + b|z|^2\right)^{-\frac{1}{p-1}},$$
where $b = b(\alpha, p, N)$ is a positive constant which goes to $+\infty$ as $\alpha \to 0$.  Thus, their result does not recover the result obtained in \cite{BKnon94} and \cite{MZdm97} for the case $\alpha = 0$ where $\mu = \frac{1}{2}$ and $b > 0$; on the contrary, our result obtained for equation \eqref{equ:problem} recovers the case $\alpha = 0$.
\end{remark}

\begin{remark} By the space translation invariance of equation \eqref{equ:problem}, the result of Theorem \ref{theo:1} can be stated for any arbitrary considered point $x_0 \in \RN$ as a blowup point of the solution, in particular, we have for all $t \in [0, T)$,
$$\left\|(T-t)e^{U(x, t)} + \left(1 + \frac{|x - x_0|^2}{(4 + 4\alpha) (T-t)|\ln (T-t)|}\right)^{-1}\right\|_{W^{1,\infty}(\RN)} \leq \frac{C}{\sqrt{|\ln (T-t)|}},$$
and the final blowup profile satisfies
$$U_\alpha^*(x) \sim \ln \left(\frac{(8 + 8\alpha)|\ln |x - x_0||}{|x - x_0|^2} \right)\quad \text{as}\;\; |x - x_0| \to 0.$$ 
\end{remark}

\begin{remark} We conjecture that the identity \eqref{equ:finalprofileUlpha} also holds after differentiating, namely that 
$$|\nabla U^*_\alpha(x)| \sim \frac{c^*}{|x|}, \quad \text{as} \quad |x| \to 0,$$
where $c^*$ is some positive constant. We strongly believe that with some better refined estimate of \eqref{equ:behTh1}, we could obtain such a final profile for the gradient solution. Unfortunately, we are only able to derive in this work a weaker result given in \eqref{estNalUalpha}.
\end{remark}

\begin{remark}  In comparison with the work in \cite{GVjde96}, where the authors deal with the study of the behavior of the profile near blowup of the following quasilinear parabolic equation 
\begin{equation}\label{equ:equGV}
\partial_t U = \Delta U  + |\nabla U|^2 + U^\beta, \quad \beta > 1,
\end{equation}
our problem can be considered as the limit case of equation \eqref{equ:equGV} as $\beta \to +\infty$, and our result obtained $(ii)$ of Theorem \ref{theo:1} is quite different from that of \cite{GVjde96}. More precisely, they showed that in the case of the single point blowup in the equation \eqref{equ:equGV} and under the radially symmetric, i.e $U(x,t) = U(r, t)$ with $r=|x|$, and nonincreasing assumption (see Theorem 8, page 36 in \cite{GVjde96}), the profile near blowup is given by 
$$\forall \beta > 2, \quad  U(r,T) \sim C_*r^{-\frac{2}{\beta - 2}} \quad \text{as}\;\; r \to 0,$$
where $C_* = C_*(\beta)$ is a positive appreciated  constant.
\end{remark}

As in \cite{MZdm97} and \cite{TZpre15} (see also \cite{ZAAaihp02}, \cite{MZjfa08}), it is possible to make the interpretation of the finite-dimensional variable in terms of the blowup time and the blowup point. This allows us to derive the stability of the profile $\Phi_\alpha$ in Theorem \ref{theo:1} with respect to perturbations of the initial data. More precisely, we have the following:
\begin{theorem}[\textbf{Stability of the blowup profile \eqref{equ:behTh1}}] \label{theo:2}
Let us denote by $\hat{U}(x,t)$ the solution constructed in Theorem \ref{theo:1} and by $\hat{T}$ its blowup time. Then, there exists a neighborhood $\mathcal{V}_0$ of $\hat{U}(x,0)$ in $\mathcal{H}_{a}$ defined in \eqref{def:Ha1a2} such that for any $U_0 \in \mathcal{V}_0$, equation \eqref{equ:problem} has a unique solution $U(x, t)$ with initial data $U_0$, and $U(x,t)$ blows up in finite time $T(U_0)$ at point $a(U_0)$. Moreover, estimate \eqref{equ:behTh1} is satisfied by $U(x-a,t)$ and
$$T(U_0) \to \hat{T}, \quad a(U_0) \to 0 \quad \text{as $U_0 \to \hat{U}_0$ in $\mathcal{H}_{a}$}.$$
\end{theorem}
\begin{remark}
As in \cite{Brejde92}, we conjecture that the blowup pattern just described in part $(i)$ of Theorem \ref{theo:1} is generic, i.e there exists an open, everywhere dense set $\mathcal{V}_0$ of initial data whose corresponding solutions either converge to a steady state, or blowup in finite time at a single point, according to the estimate \eqref{equ:behTh1}. Up to our knowledge, the only proof for the genericity is given by Herrero and Vel\'azquez \cite{HVasnsp92} for equation \eqref{equ:sheGra} without the nonlinear gradient term ($\alpha = 0$) in the one dimensional case.
\end{remark}

\begin{remark}  We will not give the proof of Theorem \ref{theo:2} because the stability result follows from the reduction to a finite dimensional case as in \cite{MZdm97} (see Theorem 2 and its proof in Section 4) and \cite{TZpre15} (see Theorem 9 and its proof in Section 6) with the same argument. Hence, we only prove the exsitence result (Theorem \ref{theo:1}) and kindly refer the reader to \cite{MZdm97} and \cite{TZpre15} for the proof of the stability.
\end{remark}

Let us now give the main ideas of the proof of Theorem \ref{theo:1}. Note that our proof is quite different from that of \cite{Brejde92} (see also \cite{Breiumj90}) treated for equation \eqref{eq:expU}. Here, we follow the method developed by Bricmont and Kupiainen \cite{BKnon94}, and modified by Merle and Zaag \cite{MZdm97} for the construction of a stable blowup solution to equation \eqref{eq:Up}. Note that the method of \cite{MZdm97} has been proved to be successful for various situations including parabolic and hyperbolic equations. For the parabolic equations, we would like to mention the works by Masmoudi and Zaag \cite{MZjfa08} (see also the earlier work by Zaag \cite{ZAAihn98}) for the complex Ginzburg-Landau equation with no gradient structure, by Nguyen and Zaag \cite{NZsns16}, \cite{NZens16} for a logarithmically perturbed nonlinear heat equation and for a refined blowup profile for equation \eqref{eq:Up}, or by Nouaili and Zaag \cite{NZcpde15} for a non-variational complex-valued semilinear heat equation, by Ghoul, Nguyen and Zaag \cite{GNZpre16c} for a non-variational parabolic system.  There are also the cases for the construction of multi-solitons for the semilinear wave equation in one space dimension by C\^ote and Zaag \cite{CZcpam13}. 

Our goal is to construct for equation \eqref{equ:problem} a solution $U(x, t)$ which blows up in finite time $T$ and verifies the behaviors \eqref{equ:behTh1}. The proof is performed in the framework of the similarity variables defined in \eqref{def:changeVar}. By the space translation invariance of equation \eqref{equ:problem}, we may assume $a = 0$ in \eqref{def:changeVar} and write $W = W_a$ for simplicity of  the notation. We recall from \eqref{eq:perSHE} that when $r = 2$, $W$ solves the equation
\begin{equation}\label{equ:W}
\partial_sW = \Delta W - \frac{y}{2}\cdot \nabla W + \alpha|\nabla W|^2 + e^W - 1.
\end{equation}
If we introduce 
$$Z = e^W,$$
then $Z$ solves
\begin{equation}\label{equ:Zys}
\partial_s Z = \Delta Z - \frac{y}{2}\cdot \nabla Z + (\alpha - 1)\frac{|\nabla Z|^2}{Z} + Z^2 - Z.
\end{equation}
Constructing a solution for \eqref{equ:problem} satisfying \eqref{equ:behTh1} reduces to the construction of a solution for \eqref{equ:Zys} such that 
\begin{equation}\label{def:Qys}
\begin{array}{ll}
Q(y, s) &= e^{W(y,s)} - e^{\Phi_\alpha\left(\frac{y}{\sqrt{s}}\right)} \equiv Z(y,s) - \Gamma_\alpha(y,s)   \to 0,
\end{array}\quad \text{as}\; s \to +\infty,
\end{equation}
where $Q$ satisfies the equation 
\begin{equation}\label{eq:Q_i}
\partial_s Q = \big(\Ls + V(y,s)\big)Q + Q^2 + (\alpha - 1)\frac{|\nabla Q + \nabla \Gamma_\alpha|^2}{Q + \Gamma_\alpha} + R(y,s), 
\end{equation}
and 
$$\Ls = \Delta - \frac{y}{2}\cdot \nabla + 1, \quad V = 2(\Gamma_\alpha - 1), \quad \|R(y,s)\|_{L^\infty} \leq \frac{C}{s}.$$
Satisfying such a property is guaranteed by a condition that $Q(s)$ belongs to some set $\mathcal{V}_{A,K_0}(s)$ which shrinks to $0$ as $s \to +\infty$ (see item $(i)$ in Definition \ref{def:St} below for an example). Let us insist on the fact that we do not work with equation \eqref{equ:W}, but with equation \eqref{equ:Zys} (see explaination below). Since the linearization of equation \eqref{equ:Zys} around the profile $e^{\Phi_\alpha} \equiv \Gamma_\alpha$ gives $N+1$ positive modes, $\frac{N(N+1)}{2}$ zero modes, and an infinite dimensional negative part, we can use the method of \cite{BKnon94} and \cite{MZdm97} which relies on two arguments:
\begin{itemize}
\item[-] The use of the bounding effect of the heat kernel to reduce the problem of the control of $Q(s)$ in $\mathcal{V}_{A,K_0}(s)$ to the control of its positive modes (in fact, we control a modified version of $Q$ in $\mathcal{V}_{A,K_0}$; see Proposition \ref{prop:redu}). Let us insist on the fact that if we linearize equation \eqref{equ:W} around $\Phi_\alpha$, say that $q = W - \Phi_\alpha$, we could not able to control $q$ outside the \textit{blowup region} because in this region the linear part of the equation satisfied by $q$ does not have a fully negative spectrum. In the contrary, the spectrum of the linear operator of the equation satisfied by $Q$ is negative outside the \textit{blowup region}, which makes the control of $Q$ in that region easily.\\

\item[-] The control of the $(N+1)$ positive modes thanks to a topological argument based on index theory (see the arguments at page \pageref{step2:topoarg}). 
\end{itemize}
(Note that the topological argument is also used in \cite{Brejde92} to solve the finite dimensional problem, however, the approach of the reduction of the problem to a finite dimensional one in our proof is totally different from that in \cite{Brejde92}).\\

Since the gradient term in the equation \eqref{equ:problem} in the critical case contributes to the change of the blowup profile and given the fact that the blowup profile $\Phi_\alpha$ is different from the one considered in \cite{BKnon94} and \cite{MZdm97}, our proof truly requires crucial modifications of the methods of \cite{BKnon94} and \cite{MZdm97} and special arguments in order to handle the nonlinear gradient term. These modifications lay in the following places:
\begin{itemize}
\item[(i)] We no longer work with the linearization of equation \eqref{equ:W} around $\Phi_\alpha$, but instead the equation \eqref{eq:Q_i} satisfied by $Q(y,s)$.  Note that linearizing \eqref{equ:W} around the profile $\Phi_\alpha$ generates the linear operator  $\Ls + \tilde V$ which behaves similarly as in the power case (equation \eqref{equ:sheGra}) inside the \textit{blowup region}, but there is a quite difference outside the \textit{blowup region}. Indeed, Bricmont and Kupiainen's approach for equation \eqref{equ:sheGra} gives that the spectrum of $\Ls + \tilde V$ in the outer region is controlled by $-\frac{1}{p-1} + \epsilon$ which can be made negative by taking $\epsilon$ small enough. Whereas, the same estimate for equation \eqref{equ:problem} gives that the spectrum of  $\Ls + \tilde V$ is controlled by $0 + \epsilon$, which can never be made negative. In order to overcome this difficulty, we no longer work with the linearization of equation \eqref{equ:W} around $\Phi_\alpha$, but instead the equation \eqref{eq:Q_i} satisfied by $Q$. In fact, the spectrum  of the linear part of equation \eqref{eq:Q_i} is fully negative in the outer region, which make it easy to control. However, this manner gives additionally a term of the form $\frac{|\nabla Q + \nabla \Gamma_\alpha|^2}{Q + \Gamma_\alpha}$ (see \eqref{eq:Q_i}) which is needed new ideas to achieve the control.
\item[(ii)] Defining the shrinking set $\mathcal{S}^*$ (see Definition \ref{def:St} below) to trap the solution. Note that our definition of $\mathcal{S}^*$ is defferent from the one in \cite{MZdm97} designed for the standard nonlinear heat equation \eqref{equ:sheGra} with $\alpha = 0$. Note also that equation \eqref{eq:Q_i} is almost the same as in \cite{MZdm97}, except for the nonlinear gradient term which causes serious difficulties in the analysis. In \cite{MZdm97}, the authors introduced estimates of $Q$ in the \textit{blowup region} $|y| \leq K_0\sqrt{s}$, and in the regular region $|y| \geq K_0\sqrt s$. However, the estimates in the region $|y| \geq K_0\sqrt{s}$ imply smallness of $Q$ only, and do not allow any control of the nonlinear gradient term in this region. In other words, the analysis based on the method of \cite{MZdm97}, that is to estimate the solution in the $z = \frac{y}{\sqrt{s}}$ variable is not sufficient and must be improved. In particular, we introduce estimates of $Q$ in three regions in a different variable scale, which follows the approach of \cite{MZnon97} using for a finite time quenching problem of vortex reconnection with the boundary. More precisely, in the \textit{blowup region}, i.e. $|x| \leq K_0\sqrt{(T-t)|\ln (T-t)|}$, we do an asymptotic analysis around the profile $\Gamma_\alpha$ through equation \eqref{eq:Q_i}. In the intermediate region, i.e. $\frac{K_0}{4}\sqrt{(T-t)|\ln(T-t)|} \leq |x| \leq \epsilon_0$, we control the solution by using classical parabolic estimates through introducing a rescaled function of $U$ (see \eqref{def:Uc} below). In the regular region, i.e. $|x| \geq \frac{\epsilon_0}{4}$, we estimate directly $U$. The estimates in the intermediate and regular regions allow us to control the nonlinear gradient term appearing in equation \eqref{eq:Q_i}. 
\end{itemize}

The derivation of the final profile $U_\alpha^*(x)$ stated in part $(ii)$ of Theorem \ref{theo:1} uses the method of \cite{ZAAihn98} and \cite{Mercpam92}. In particular, we need to establish a new \textit{no blowup under some threshold} criterion for a parabolic inequality with a nonlinear gradient term, whose proof follows ideas given in \cite{GKcpam89} (see Proposition \ref{prop:noblowup}).\\

\noindent The organization of the rest of this paper is as follows:\\
- In Subsection \ref{sec:formalapproach}, we first explain formally how we obtain the profile $\Phi_\alpha$ and give a suggestion for an appreciated profile to be linearized around. In Subsection \ref{sec:transproble}, we give a formulation of the problem in order to justify the formal argument.\\
- In Section \ref{sec:existence}, we give all the arguments of the proof of part $(i)$ of Theorem \ref{theo:1} assuming technical results, which are left to the next section. \\
- In Section \ref{sec:reductofn}, we give the proof of the technical results used in the existence's proof, that is the proof of Proposition \ref{prop:redu}. We divide its proof in two subsections. In Subsection \ref{sec:apri}, we derive \textit{a priori estimates} of $U(t)$ in the set $\Sc^*(t)$. In Subsection \ref{sec:con}, we show that all the estimates given in Definition \ref{def:St} of $\Sc^*$ can be improved, except for the positive modes, which concludes the proof of the reduction to a finite dimensional one.\\
- In Appendix \ref{sec:preinti}, we give properties of the shrinking set $\Sc^*(t)$ to trap the solution as well as properties of the initial data corresponding to the blowup solution described in Theorem \ref{theo:1}. In Appendix \ref{sec:noblowup}, we give the proof of a \textit{no blowup under some threshold} criterion for a parabolic inequality with a nonlinear gradient term, which is an ingredient in the proof of the existence of the final blowup profile.

\section{Formulation of the problem.}
\subsection{A formal analysis.}\label{sec:formalapproach}
In this subsection, we use matching asymptotics to formally derive the blowup behavior of the solution to \eqref{equ:problem}. More precisely, we will explain how to deduce the following behavior
\begin{equation}\label{equ:behU}
U(x,t) + \ln(T-t) \sim \Phi_\alpha\left(\frac{x}{\sqrt{(T-t)|\ln(T-t)|}}\right) \quad \text{as} \;\; t\to T,
\end{equation}
where $\Phi_\alpha$ is defined by \eqref{def:Phi_alpha}. 

In the similarity variables setting \eqref{def:changeVar}, justifying \eqref{equ:behU} is equivalent to showing that 
$$W(y,s) \sim \Phi_\alpha\left(\frac{y}{\sqrt{s}}\right)\quad \text{as} \;\; s \to +\infty,$$
where $W$ satisfies \eqref{equ:W}. Let us rewrite the equation on $W$ as follows:
\begin{equation}\label{equ:WL}
\partial_sW = \Ls W + \alpha|\nabla W|^2 + Q(W),
\end{equation}
where 
\begin{equation}\label{def:opL}
\Ls = \Delta - \frac{y}{2}\cdot \nabla + 1,
\end{equation}
and 
\begin{equation*}
Q(W) = e^W - 1 - W.
\end{equation*}
Note that we have 
$$\big|Q(W) - \frac{1}{2}W^2 \big| \leq C|W|^3 \quad \text{as}\;\; W \to 0.$$
The linear operator $\Ls$ is self-adjoint in $L^2_\rho(\RN)$, where $L^2_\rho(\RN)$ is the weighted $L^2$ space associated with the weight $\rho$ defined by 
$$\rho(y) = \frac{1}{(4\pi)^{N/2}}e^{-\frac{|y|^2}{4}}.$$
It can be shown that the spectrum of $\Ls$ is explicitly given by (see \cite{FKcpam92} for instance)
$$\text{spec}(\Ls) = \left\{1 - \frac{n}{2},n \in \mathbb{N}\right\}.$$
For $N = 1$, all the eigenvalues are simple and the eigenfunctions are dilations of Hermite polynomials: the eigenvalue $1 - \frac{n}{2}$ corresponds to the following eigenfunction:
\begin{equation}
h_n(y) = \sum_{i = 0}^{\big[n/2\big]} \frac{n!}{i!(n - 2i)!}(-1)^iy^{n - 2i}.
\end{equation}
The first three eigenfunctions are 
$$h_0(y) = 1, \quad h_1(y) = y, \quad h_2(y) = y^2 - 2.$$
Notice that $h_n$ satisfies
\begin{equation}\label{equ:inthnhm}
\int_{\Rb} h_n(\xi) h_m(\xi) \rho(\xi) d\xi = 2^nn!\delta_{n,m}, \quad \Ls h_n = \big(1 - \frac{n}{2}\big)h_n.
\end{equation}
We also introduce $k_n(y) = \|h_n\|^{-2}_{L^2_\rho}h_n(y)$. In higher dimensions, the eigenfuntions are formed by taking products of the polynomials $\{h_k\}_{k = 0}^\infty$, namely that for $n = (n_1, \cdots, n_N) \in \mathbb{N}^N$, the eigenfunction corresponding to $1 - \frac{|n|}{2}$ ($|n| = n_1 + \cdots + n_N$) is 
\begin{equation}\label{def:hn1nN}
h_n(y) = h_{n_1}(y_1)\cdots h_{n_N}(y_N).
\end{equation}

Since the eigenfunctions of $\Ls$ span the whole space $L^2_\rho(\RN)$, we can expand the solution $W$ as follows: 
$$W(y,s) = \sum_{m \in \mathbb{N}^N}^\infty W_m(s)h_m(y).$$
For simplicity, let us assume that the solution is radially symmetric in $y$. Since $h_m$ with $m \geq  3$ correspond to negative eigenvalues of $\Ls$, we may consider that 
\begin{equation}\label{equ:expandW_02}
W(y,s) = W_0(s) + W_2(s)(|y|^2 - 2N),
\end{equation}
with $W_0(s), W_2(s) \to 0$ as $s \to +\infty$ (for simplicity in the notation, we write $W_2(s)$ instead of $W_{(2, 0, \cdots, 0)}(s)$).

Writing $\alpha|\nabla W|^2 = 4\alpha (h_2 + 2N h_0)W_2^2$, multiplying equation \eqref{equ:WL} by $k_0\rho$ and $k_2\rho$ respectively and integrating over $\RN$, we derive the following ODE system for $W_0$ and $W_2$:
\begin{align}
W_0' &= W_0 + \frac{1}{2}W_0^2 + N(4 + 8\alpha)W_2^2 + \mathcal{O}(|W_0|^3 + |W_2|^3)\label{equ:W_0}\\
W_2' &= 0 + W_0W_2 + (4 + 4\alpha)W_2^2+ \mathcal{O}(|W_0|^3 + |W_2|^3).\label{equ:W_2}
\end{align}
Assuming that 
\begin{equation}\label{equ:assW0W2}
|W_0(s)|  = o(|W_2(s)|)\quad \text{as} \quad s \to +\infty,
\end{equation}
we then rewrite equation \eqref{equ:W_2} as follows:
$$W_2' = (4 + 4\alpha)W_2^2 + o(|W_2|^2) \quad \text{as} \quad s \to +\infty,$$
which yields
$$W_2(s) = -\frac{1}{(4 + 4\alpha)s} + o\left(\frac{1}{s}\right) \quad \text{as} \quad s \to +\infty.$$
From equation \eqref{equ:W_0} and the above formula of $W_2$, we have $W_0' = W_0 + \mathcal{O}\left(\frac{1}{s^2}\right)$, which gives
$$|W_0(s)| = \mathcal{O}\left(\frac{1}{s^2}\right) \quad \text{as}\quad s \to +\infty.$$
From the above formulas of $W_0$ and $W_2$, we have by equation \eqref{equ:W_2},
$$W_2' = (4 + 4\alpha)W_2^2 + \mathcal{O}\left(\frac{1}{s^3}\right) \quad \text{as} \quad s \to +\infty,$$
hence, 
$$W_2(s) = -\frac{1}{(4 + 4\alpha)s} + \mathcal{O}\left(\frac{\ln s}{s^2}\right) \quad \text{as}\quad s \to +\infty.$$
Such $W_0$ and $W_2$ are compatible with the assumption \eqref{equ:assW0W2}. Therefore, we write by \eqref{equ:expandW_02},
\begin{equation}\label{equ:expandWinL2rho}
W(y,s) = -\frac{1}{4 + 4\alpha}\frac{y^2}{s} + \frac{N}{(2 + 2\alpha)s} + \mathcal{O}\left(\frac{\ln s}{s^2}\right) \quad \text{as}\quad s \to +\infty,
\end{equation}
in $L^2_\rho$, and also uniformly in compact sets by standard parabolic regularity.

Since \eqref{equ:expandWinL2rho} provides a relevant variable for blowup, namely $z = \frac{y}{\sqrt{s}}$, we then try to search formally solutions of \eqref{equ:W} of the form
\begin{equation}\label{equ:expantildeW}
\tilde{W}(y,s) = \sum_{j = 0}^\infty \frac{1}{s^j}w_j(z), \quad z = \frac{y}{\sqrt{s}},
\end{equation}
and compare elements of order $\frac{1}{s^j}$. For $j = 0$ and $j=1$, we find that
\begin{equation}\label{equ:w_0}
-\frac{z}{2}\cdot\nabla w_0(z) + e^{w_0(z)} - 1 = 0,
\end{equation}
\begin{equation}\label{equ:w_1}
F(z):= \frac{z}{2}\cdot \nabla w_1(z) - e^{w_0(z)}w_1(z) - \frac{z}{2}\cdot \nabla w_0(z) - \Delta w_0(z) - \alpha|\nabla w_0(z)|^2 = 0.
\end{equation}
Recalling that $\tilde{W} \to 0$ as $s\to +\infty$ on every compact set, we naturally impose the condition 
$$w_0(0) = 0.$$
Solving equation \eqref{equ:w_0} with this condition, we then obtain by radial symmetry
$$w_0(z) = -\ln(1 + c_0|z|^2),$$
for an integration constant $c_0$. Because we want a bounded solution, then it requires $c_0 > 0$. From the Taylor expansion for $|y|$ bounded, we write
$$\tilde{W}(y,s) \sim -\ln \left(1 + c_0\frac{|y|^2}{s}\right) \sim -c_0\frac{|y|^2}{s} - \frac{c_0^2}{2}\frac{|y|^4}{s^2} + \cdots,$$
from which and \eqref{equ:expandWinL2rho}, we find that the coefficient $c_0=\frac{1}{4 + 4\alpha}$. Thus,  
$$w_0(z) = -\ln\left(1 + \frac{|z|^2}{4 + 4\alpha}\right).$$
Substituting this formula into \eqref{equ:w_1} and evaluating $F$ at $z = 0$, we obtain
$$w_1(0) = 2Nc_0 = \frac{N}{2 + 2\alpha},$$
which matches with the expansion \eqref{equ:expandWinL2rho}.

In conclusion, the first term in the expansion \eqref{equ:expantildeW} of $\tilde{W}$ is precisely the profile function $\Phi_\alpha$ as expected in \eqref{equ:behU}.

\subsection{Transformation of the problem.}\label{sec:transproble}
In this subsection, we set up the problem of constructing a solution $U$ for equation  \eqref{equ:problem} which blows up in finite time $T$ only at the origin and satisfies \eqref{equ:behTh1}. In particular, we want to prove for suitable initial data $U_0$ of \eqref{equ:problem} that 
\begin{equation}\label{equ:beU1}
\lim_{t \to T}\left\|(T-t)e^{U(z\sqrt{(T-t)|\ln(T-t)|},t)} - e^{\Phi_\alpha(z)} \right\|_{W^{1, \infty}(\RN)} = 0,
\end{equation}
where $\Phi_\alpha$ is defined by \eqref{def:Phi_alpha}. Or, in the self-similar setting \eqref{def:changeVar}, our aim is to construct initial data $W(s_0)$ such that the equation \eqref{equ:W} has a solution $W(y,s)$ defined for all $(y,s) \in \RN \times [s_0, +\infty)$ and satisfies
\begin{equation}\label{equ:beW2}
\lim_{s \to +\infty} \left\|e^{W(y,s)} - e^{\Phi_\alpha\left(\frac{y}{\sqrt{s}}\right)}\right\|_{W^{1,\infty}(\RN)} = 0.
\end{equation}
In the previous subsection, we formally obtain for $W(y,s)$ an expansion of the form $\sum_{j = 0}^{+\infty}\frac{1}{s^j}w_j(\frac{y}{\sqrt{s}})$ with $w_0 = \Phi_\alpha$ and $w_1$ satisfying \eqref{equ:w_1}. Hence, we will not linearize $W$ around $\Phi_\alpha$, but we will study the difference $W(y,s) - w_0(\frac{y}{\sqrt{s}}) - \frac{1}{s}w_1(\frac{y}{\sqrt{s}})$. Since the expression of $w_1$ is too complicated, we study instead $W(y,s) - w_0(\frac{y}{\sqrt{s}}) - \frac{1}{s}w_1(0)$. Let us introduce 
\begin{equation}\label{def:Qpsi}
Q(y,s) = e^{W(y,s)} - \psi_\alpha(y,s) \equiv Z(y,s) - \psi_\alpha(y,s) \;\; \text{where}\;\; \psi_\alpha(y,s) = e^{\frac{N}{(2 + 2\alpha)s}}e^{\Phi_\alpha\left(\frac{y}{\sqrt{s}}\right)},
\end{equation}
where $\Phi_\alpha$ is defined in \eqref{def:Phi_alpha}. Then, we see from \eqref{equ:Zys} that $Q$ satisfies
\begin{equation}\label{equ:Q}
\partial_s Q = \big(\Ls + V(y,s)\big)Q + Q^2 + G(Q) + R(y,s), 
\end{equation}
where $\Ls$ is defined by \eqref{def:opL} and
\begin{align}
V(y,s) &= 2(\psi_\alpha(y,s) - 1),\label{def:Vys}\\ 
G(Q)& = (\alpha - 1)\left[\frac{|\nabla Q + \nabla \psi_\alpha|^2}{Q + \psi_\alpha} - \frac{|\nabla \psi_\alpha|^2}{\psi_\alpha} \right],\label{def:GQ}\\
R(y,s) &= -\partial_s \psi_\alpha + \Delta \psi_\alpha - \frac{y}{2}\cdot\nabla \psi_\alpha
+\psi_\alpha^2 - \psi_\alpha + (\alpha - 1)\frac{|\nabla \psi_\alpha|^2}{\psi_\alpha}. \label{def:Rys}
\end{align}
(Note that the equation satisfied by $Q$ is almost the same as in \cite{MZdm97}, except the term $G(Q)$).
 
Satisfying  \eqref{equ:beW2} reduces to the construction of initial data $Q(y, s_0)$ such that the equation \eqref{equ:Q} has a solution $Q(y,s)$ defined for all $(y,s) \in \RN\times [s_0, +\infty)$ such that 
\begin{equation}\label{con:Qu}
\lim_{s \to +\infty} \|Q(s)\|_{W^{1,\infty}(\RN)} = 0.
\end{equation}

Our analysis uses the Duhamel formulation of equation \eqref{equ:Q}: for each $s \geq  \sigma \geq  s_0$, we have
\begin{equation}\label{for:qint}
Q(s) = \mathcal{K}(s,\sigma)Q(\sigma) + \int_\sigma^s \mathcal{K}(s, \tau)\left[Q^2(\tau) + G(Q(\tau)) + R(\tau)\right] d\tau,
\end{equation}
where $\mathcal{K}$ is the fundamental solution of the linear operator $\Ls + V$ defined for each $\sigma > 0$ and $s \geq  \sigma$ by
\begin{equation} \label{def:kernel}
\partial_s\mathcal{K}(s,\sigma) = (\Ls + V)\mathcal{K}(s,\sigma), \quad \mathcal{K}(\sigma, \sigma) = \text{Identity}.
\end{equation}
Since we want to construct for \eqref{equ:Q} a solution $Q$ satisfying  \eqref{con:Qu} and the fact that 
$\|R(s)\|_{L^\infty} \leq \frac{C}{s}$ for $s$ large, it is then reasonable to think that the dynamics of equation \eqref{equ:Q} are influenced by the linear part, namely $\Ls + V$. The properties of the  self-adjoint operator $\Ls$ are given in the previous subsection. In particular, $\Ls$ is predominant on all the modes, except on the null modes where the terms $VQ$ and $G(Q)$ will play a crucial role. As for potential $V$, it has two fundamental properties which will strongly influence our strategy:
\begin{itemize}
\item[(i)] $\|V(s)\|_{L^2_\rho(\RN)} \to 0$ as $s \to +\infty$. In practice, the effect of $V$ in the \textit{blowup region} $\{|y|\leq K_0\sqrt{s}\}$ is regarded as a perturbation of the effect of $\Ls$ (except for the null mode).
\item[(ii)] outside of the \textit{blowup region}, we have the following property: for all $\epsilon > 0$, there exist $K_\epsilon > 0$ and $s_\epsilon$ such that
$$|V(y,s) - (-1)| \leq \epsilon - 1, \quad \forall s \geq  s_\epsilon, \; \forall |y| \geq  K_\epsilon \sqrt{s},$$
which $1$ is the largest eigenvalue of the operator $\Ls$. Thus, the spectrum of the linear operator $\Ls + V$ is fully negative, hence, the control of $Q$ in $L^\infty$ ouside the \textit{blowup region} will be done without difficulties. Note that linearizing equation \eqref{equ:W} around $\Phi_\alpha$, say $q = W - \Phi_\alpha$, generates the linear operator $\Ls + \tilde{V}$ whose spectrum in the outer region is fully positive. This is the major reason why we do not work with $q$, but with $Q = Z - \psi_\alpha$. \\ 
\end{itemize}

Since the behavior of the potential $V(y,s)$ inside and outside of the \textit{blowup region} is different, let us decompose $Q$ as follows:
\begin{equation}\label{def:qbqe}
Q(y,s) = \chi(y,s) Q(y,s) + (1 - \chi(y,s))Q(y,s) \equiv Q_b(y,s) + Q_e(y,s),
\end{equation}
where
\begin{equation}\label{def:chi}
\chi(y,s)= \chi_0\left(\frac{|y|}{K_0\sqrt{s}}\right),
\end{equation}
and $\chi_0 \in \mathcal{C}_0^\infty([0,+\infty))$ with $\text{supp}(\chi_0) \subset [0,2]$ and $\chi_0 \equiv 1$ on $[0,1]$, and $K_0 > 0$ is to be fixed large enough. Note that $supp (Q_b(s)) \subset \mathbf{B}(0,2K_0\sqrt{s})$ and $supp (Q_e(s)) \subset \mathbb{R}^N\setminus\mathbf{B}(0,K_0\sqrt{s})$.

Since the eigenfunctions of $\Ls$ span the whole space $L^2_\rho(\mathbb{R}^N)$, let us write 
\begin{align}
Q_b(y,s) &= Q_0(s) + Q_1(s)\cdot y + y^TQ_2(s)y - 2tr(Q_2(s)) + Q_{-}(y,s) \nonumber \\
& = Q_0(s) + Q_1(s)\cdot y + Q_\bot(y,s),\label{exp:Qb}
\end{align}
where
$$Q_{0}(s) = P_0(Q_b)(y,s), \quad Q_1(s)\cdot y = P_1 (Q_b)(y,s),$$
$$Q_{-}(y,s) = P_-(Q_b)(y,s) = \sum_{m \geq  3} P_m(Q_b)(y,s),$$
$$Q_{\bot}(y,s) = P_\bot(Q_b)(y,s) = \sum_{m \geq  2} P_m(Q_b)(y,s),$$
and $P_m$ is the projector on the eigenspace corresponding to the eigenvalue $1 - \frac{m}{2}$ defined by
\begin{equation}\label{def:Projector}
P_m(Q_b)(y,s) = \sum_{\beta \in \mathbb{N}^N, |\beta| = m}  \frac{h_\beta(y)}{\|h_\beta\|^2_{L^2_\rho}}\int_{\mathbb{R}^N} h_\beta(y)Q_b(y,s)\rho(y)dy,
\end{equation}
where $h_\beta$ is defined by \eqref{def:hn1nN}, and $Q_{2}(s)$ is a symmetric $(N \times N)$ matrix defined by
\begin{equation}\label{def:vb2}
Q_{2}(s) = \int_{\mathbb{R}^N}Q_b(y,s)\mathcal{M}(y)\rho(y)dy,
\end{equation}
and 
\begin{equation}\label{def:M2}
\mathcal{M}(y) = \left\{\frac{1}{8}y_iy_j - \frac{1}{4}\delta_{ij} \right\}_{1\leq i,j \leq N}.
\end{equation}
According to \eqref{def:qbqe} and \eqref{exp:Qb}, we have 
\begin{equation}\label{decomqQ}
Q(y,s) = Q_0(s) + Q_1(s)\cdot y + y^TQ_2(s)y - 2tr(Q_2(s)) + Q_{-}(y,s) + Q_e(y,s),
\end{equation}
and 
\begin{equation}\label{decomQ2}
Q(y,s) = Q_0(s) + Q_1(s)\cdot y + Q_\bot(y,s) + Q_e(y,s).
\end{equation}
The reader should keep in mind that $Q_m,\, m = 0, 1, 2$, $Q_-$ and $Q_\bot$ are the coordinates of $Q_b$ and not those of $Q$.

\section{Proof of the existence without technical details.}\label{sec:existence}
This section is devoted to the proof of Theorem \ref{theo:1}, that is the existence of the  solution $U(t)$ of equation \eqref{equ:problem} satisfying 
\begin{equation}\label{equ:goalLim}
\lim_{t \to T} \left\|(T-t)e^{U(x,t)} - e^{\Phi_\alpha\left(\frac{x}{\sqrt{(T-t)|\ln(T-t)|}}\right)}\right\|_{W^{1,\infty}(\RN)} = 0,
\end{equation}
and 
\begin{equation}\label{eq:goal2}
 \lim_{t \to T}U(x,t) = U^*_\alpha(x),
\end{equation}
where $\Phi_\alpha$ is defined by \eqref{def:Phi_alpha} and $U^*_\alpha$ behaves as \eqref{equ:finalprofileUlpha} as $x \to 0$. According to the transformation \eqref{def:Qpsi}, we see that the proof of \eqref{equ:goalLim} is equivalent to proving the existence of the solution $Q(s)$ of equation \eqref{equ:Q} such that
\begin{equation}\label{eq:goalQ}
\lim_{s \to +\infty} \|Q(s)\|_{W^{1, \infty}(\RN)} = 0.
\end{equation}   
We shall give all the arguments of the proof for \eqref{eq:goalQ} and \eqref{eq:goal2} assuming technical results which are left to the following sections.\\

In order to prove \eqref{eq:goalQ}, we use ideas given in Merle and Zaag \cite{MZnon97} where the authors suggested a modification of the argument of \cite{MZdm97} for the standard nonlinear heat equation \eqref{equ:sheGra} with $\alpha = 0$. In particular, we shall control the solution in three different regions covering $\RN$, defined as follows: for $K_0 > 0$, $\epsilon_0 > 0$ and $t \in [0,T)$, we set 
\begin{align*}
\Dc_1(t) &= \left\{x\;  \Big\vert \; |x| \leq K_0 \sqrt{|\ln(T-t)|(T-t)} \right\}\\
&\quad \equiv \left\{x \;\big\vert \; |y| \leq K_0 \sqrt{s}\right\} \equiv \left\{ x \; \Big\vert \; |z| \leq K_0\right\},\\
\Dc_2(t) &= \left\{x \; \Big\vert \; \frac{K_0}{4} \sqrt{|\ln(T-t)|(T-t)} \leq |x| \leq \epsilon_0 \right\}\\
&\quad  \equiv \left\{x \; \Big\vert\; \frac{K_0}{4} \sqrt{s} \leq |y| \leq \epsilon_0e^{\frac{s}{2}}\right\} \equiv \left\{ x \; \Big \vert\; \frac{K_0}{4} \leq |z| \leq  \frac{\epsilon_0}{\sqrt s }e^{\frac s2} \right\},\\
\Dc_3(t) &= \left\{x\;  \Big\vert \; |x| \geq \frac{\epsilon_0}{4} \right\} \equiv \left\{x \;\big\vert \; |y| \geq \frac{\epsilon_0}{4}e^{\frac{s}{2}}\right\} \equiv \left\{ x \; \Big\vert \; |z| \geq \frac{\epsilon_0}{4 \sqrt s}e^{\frac{s}{2}} \right\}.
\end{align*}

- In $\Dc_1$, the \textit{blowup region} of $U$, we make the change of variables \eqref{def:Qpsi}, resulting in equation \eqref{equ:Q}, to do an asymptotic analysis around the profile $e^{\Phi_\alpha(y/\sqrt s)}$ according to the decomposition \eqref{decomqQ} and \eqref{decomQ2}.

- In the intermediate region $\Dc_2$, we control $U$ by using classical parabolic estimates on $\Uc$, a rescaled function of $U$ defined for $x \ne 0$ by 
\begin{equation}\label{def:Uc}
\mathcal{U}(x, \xi, \tau) = \ln \big(T - t(x)\big) + U\big(x + \xi\sqrt{T - t(x)}, t(x) + \tau (T - t(x))\big),
\end{equation}
where  $t(x)$ is uniquely defined for $|x|$ sufficiently small by 
\begin{equation}\label{def:tx}
|x| = \frac{K_0}{4}\sqrt{(T - t(x))|\ln (T - t(x))|} \equiv \frac{K_0}{4}\sqrt{\theta(x)|\ln \theta(x)|}, 
\end{equation}
with 
\begin{equation}\label{def:thetax}
\theta(x) = T - t(x).
\end{equation}
From \eqref{equ:problem}, we see that $\Uc$ satisfies  the same equation as $U$: for all $\xi \in \RN$ and $\tau \in \left[-\frac{t(x)}{T - t(x)},1 \right)$,
\begin{equation}\label{eq:Ucxt}
\partial_\tau \Uc = \Delta_\xi \Uc + \alpha|\nabla_\xi \Uc|^2 + e^\Uc.
\end{equation}
We will in fact prove that $U$ behaves for 
$$|\xi| \leq \alpha_0\sqrt{|\ln (T- t(x))|}\quad \text{and} \quad \tau \in \left[\frac{t_0  - t(x)}{T -t(x)},1\right)$$
for some $t_0 < T$ and $\alpha_0 > 0$, like the solution of 
\begin{equation}
\partial_\tau \hat \Uc = e^{\hat \Uc},
\end{equation}
subject to the initial data 
$$\hat \Uc(0) = \Phi_\alpha\left(\frac{K_0}{4}\right) = -\ln \left(1 + \frac{K_0^2/16}{4 + 4\alpha}\right),$$
namely that 
\begin{equation}\label{def:solUc}
\hat{\Uc}(\tau) = -\ln\left((1 - \tau) + \frac{K_0^2/16}{4 + 4\alpha}\right). 
\end{equation}
As we will see that the analysis in $\Dc_2$ will imply the conclusion of \eqref{eq:goal2}. 

- In $\Dc_3$, we estimate directly $U$ by using the local in time well-posedness of the Cauchy problem for equation \eqref{equ:problem}.\\

As described above, satisfying \eqref{eq:goalQ} and \eqref{eq:goal2} is guaranteed if we can show that 
\begin{equation}\label{equ:UinSt}
U(t) \in \mathcal{S}^*(t), \quad \forall t \in [t_0, T),
\end{equation}
where $\Sc^*(t)$ is precisely defined as follows:
\begin{definition}[\textbf{Definition of shrinking set to trap solutions}]\label{def:St} For all $t_0 < T$, $K_0 > 0$, $\epsilon_0 > 0$, $\alpha_0 > 0$, $A > 0$, $\delta_0 > 0$, $\eta_0 > 0$, $C_0 > 0$ and $C_0' > 0$, for all $t \in [t_0,T)$, we define $\Sc^*(t_0, K_0, \epsilon_0, \alpha_0, A, \delta_0, \eta_0, C_0, C_0',  t)$ being the set of all functions $U \in \mathcal{H}_{a}$ (see Definition \eqref{def:Ha1a2}) satisfying:

\noindent $(i)$ \textit{Estimates in $\Dc_1$:} $Q(s) \in \Vc_{A, K_0}(s)$ where $s = - \ln(T-t)$, $Q(s)$ is defined as in \eqref{def:Qpsi} and $\Vc_{A, K_0}(s)$ is the set of all functions $Q$ in $W^{1, \infty}(\RN)$ such that 
$$
|Q_0(s)| \leq \frac{A}{s^2},\quad |Q_{1,i}(s)|\leq \frac{A}{s^2},\quad |Q_{2,ij}(s)| \leq \frac{A^2 \ln s}{s^2},  \quad  \forall i,j \in \{1, \cdots, N\},$$
$$|Q_-(y,s)| \leq \frac{A}{s^2}(|y|^3 + 1),\quad |(\nabla_y Q)_\bot (y,s)| \leq \frac{A}{s^2}(|y|^3 + 1), \quad \forall y\in \RN,$$
$$\|Q_e(s)\|_{L^\infty(\RN)} \leq \frac{A^2}{\sqrt s},$$
where $Q_m$ ($m = 0, 1, 2$), $Q_-$, $Q_\bot$ and $Q_e$ are defined as in \eqref{decomqQ} and \eqref{decomQ2}.

\noindent $(ii)$ \textit{Estimates in $\Dc_2$:} For all $|x| \in \left[\frac{K_0}{4}\sqrt{|\ln(T-t)|(T-t)}, \epsilon_0\right]$, $\tau = \tau(x,t) = \frac{t - t(x)}{\theta(x)}$ and $|\xi| \leq \alpha_0 \sqrt{\ln \theta(x)}$, 
$$\left|\Uc(x,\xi, \tau) - \hat \Uc(\tau)\right| \leq \delta_0, \quad |\nabla_\xi \Uc(x, \xi, \tau)| \leq \frac{C_0}{\sqrt{|\ln \theta(x)|}}, \quad |\nabla^2_\xi \Uc(x, \xi, \tau)| \leq C_0',$$
where $\Uc$, $\hat \Uc$, $t(x)$ and $\theta(x)$ are defined in \eqref{def:Uc}, \eqref{def:solUc}, \eqref{def:tx} and \eqref{def:thetax} respectively.

\noindent $(iii)$ \textit{Estimates in $\Dc_3$:} For all $|x| \geq \frac{\epsilon_0}{4}$, 
$$|U(x,t) - U(x,t_0)| \leq \eta_0, \quad |\nabla_x U(x,t) - \nabla_x U(x,t_0)| \leq \eta_0.$$

\noindent For all $t_0 < T$, we define $\Sc^*(t_0, K_0, \epsilon_0, \alpha_0, A, \delta_0, \eta_0, C_0, C_0')$ being the set of all functions $U \in \mathcal{C}([t_0,T), \mathcal{H}_{a})$ such that  
$$U(t) \in \Sc^*(t_0, K_0, \epsilon_0, \alpha_0, A, \delta_0, \eta_0, C_0, C_0', t), \quad \forall t \in [t_0, T).$$
\end{definition}
\begin{remark} The estimates on $\nabla_y Q$ in $\Dc_1$, $\nabla_\xi \Uc$ in $\Dc_2$ and $\nabla_x U$ in $\Dc_3$ allow us to control the nonlinear gradient term $G(Q)$ appearing in equation \eqref{equ:Q}. Note that in the case when $G(Q)$ does not appear, the only estimates on $\Dc_1$ are enough to fully control the solution (see \cite{MZdm97}). Therefore, this part makes the originality of the paper. 
\end{remark}

\bigskip

We will show that the proof of Theorem \ref{theo:1} reduces to find suitable parameters $t_0 < T$, $K_0$, $\epsilon_0$, $\alpha_0$, $A$, $\delta_0$, $\eta_0$, $C_0$ and $U_0 \in \mathcal{H}_{a}$ so that the solution $U$ of equation \eqref{equ:problem} with data $U(t_0) = U_0$ belongs to $\Sc^*(t_0, K_0,\epsilon_0, \alpha_0, A, \delta_0, \eta_0, C_0, C_0', t_0)$. As a matter of fact, through \textit{a priori estimate}, we will show that the control of $U(t)$ in $\Sc^*(t_0, K_0,\epsilon_0, \alpha_0, A, \delta_0, \eta_0, C_0, C_0', t)$ for $t \in [t_0, T)$ reduces to the control of $(Q_0, Q_1)(s)$ in 
\begin{equation}\label{def:VcA}
\hat\Vc_A(s) = \left[-\frac{A}{s^2}, \frac{A}{s^2}\right]^{N + 1},
\end{equation}
for $s \geq -\ln (T - t_0)$ (recall that $Q_0$ and $Q_1$ correspond to expanding eigenvalues in the $Q$ variable \eqref{def:Qpsi}). Hence, we will consider initial data $U_0$ depending on $(N+1)$ parameters $(d_0, d_1) \in \Rb \times \RN$ of the form:
\begin{align}
&U_{d_0,d_1}(x, t_0) = \hat U^*(x)(1 - \chi_1(x,t_0))\label{def:U0d0d1}\\
& + \left\{s_0 + \ln \left[ \frac{A}{s_0^2}\left(d_0 + d_1\cdot x e^{\frac{s_0}{2}}\right)\chi\left(16xe^{\frac{s_0}{2}},s_0\right)  + \psi_\alpha\left(xe^{\frac{s_0}{2}}, s_0\right)\right] \right\} \chi_1(x,t_0),\nonumber
\end{align}
where $s_0 = -\ln(T-t_0)$, $\psi_\alpha$ and $\chi$ are defined in \eqref{def:Qpsi} and \eqref{def:chi}, 
\begin{equation}\label{def:chi1}
\chi_1(x,t_0) = \chi_0 \left(\frac{|x|}{|\ln(T-t_0)|\sqrt{T-t_0}}\right),
\end{equation}
with $\chi_0$ being defined right before \eqref{def:chi}, and $\hat U^* \in \mathcal{C}^\infty(\RN \backslash \{0\})$ is defined by 
\begin{equation}\label{def:Uhatstar}
\hat U^*(x) = \left\{ \begin{array}{ll}
 \ln\left(\frac{(8+8\alpha)|\ln|x||}{|x|^2} \right) &\quad \text{for}\quad |x| \leq C(a, \alpha),\\
 -\ln\left(1 + a |x|^2\right) &\quad \text{for} \quad |x|\geq 1.
\end{array}
 \right.
\end{equation}

Note that $\chi\left(16xe^{\frac{s_0}{2}},s_0\right)(1 - \chi_1(x,t_0)) = 0$ for $s_0$ large, hence, the initial data \eqref{def:U0d0d1} has an equivalence in the $Q$ variable \eqref{def:Qpsi},
\begin{equation}\label{def:barQs0}
\bar Q_{d_0,d_1}(y) = Q_{d_0,d_1}(y,s_0) = \frac{A}{s_0^2}(d_0 + d_1\cdot y)\chi(16y, s_0)e^{\chi_1(x, t_0)}.
\end{equation}
In what follows, the solution of equation \eqref{equ:problem} with initial data \eqref{def:U0d0d1} will be denote by $U_{d_0,d_1}(x,t)$ or $U(x,t)$ when there is no ambiguity. We also write $\Sc^*(t)$ and $Q(y,s)$  instead of \\
$\Sc^*(t_0, K_0,\epsilon_0, \alpha_0, A, \delta_0, \eta_0, C_0, C_0', t)$ and $Q_{d_0,d_1}(y,s)$ (the solution of equation \eqref{equ:Q} with initial data \eqref{def:barQs0}) for simplicity. 

We aim at proving the following central proposition which implies Theorem \ref{theo:1}:
\begin{proposition}[\textbf{Existence of a solution of equation \eqref{equ:problem} trapped in $\mathcal{S}^*(t)$}]\label{prop:UinSt} We can choose parameters $t_0 < T$, $K_0$, $\epsilon_0$, $\alpha_0$, $A$, $\delta_0$, $\eta_0$ and $C_0$ such that the following holds: there exists $(d_0, d_1) \in \Rb \times \RN$ such that if $U(x,t)$ is the solution of \eqref{equ:problem} with initial data at $t = t_0$ given by \eqref{def:U0d0d1}, then $U(x,t)$ exists for all $(x,t) \in \RN \times [t_0,T)$ and satisfies
$$U(t) \in \Sc^*(t_0, K_0,\epsilon_0, \alpha_0, A, \delta_0, \eta_0, C_0, C_0', t), \quad \forall t \in [t_0,T).$$

\end{proposition}

\medskip

Before going to the proof of Proposition \ref{prop:UinSt}, let us first make sure that the initial data \eqref{def:U0d0d1} starts in $\Sc^*(t_0)$ by selecting the good parameters $(d_0,d_1)$. More precisely, we have the following:

\begin{proposition}[\textbf{Properties of the initial data \eqref{def:U0d0d1}}]\label{prop:PropInitialdata} 
There exists $K_{0,1} > 0$ such that for each $K_0 \geq K_{0,1}$ and $\delta_{0,1} > 0$, there exist $\alpha_{0,1}(K_0, \delta_{0,1}) > 0$ and $C_{0,1}(K_0) > 0$ such that for all $\alpha_0 \in (0, \alpha_{0,1}]$, there exists $\epsilon_{0,1}(K_0, \delta_{0,1}, \alpha_0) > 0$ such that for all $ \epsilon_0 \in  (0,\epsilon_{0,1}]$ and $A \geq 1$, there exists $t_{0,1}(K_0, \delta_{0,1}, \epsilon_0, A, C_{0,1}) < T$ such that for all $t_0 \in [t_{0,1}, T)$,

\noindent $(I)$ There exists a subset 
$$\mathcal{D}_{t_0, A} \subset [-2,2] \times [-2, 2]^N$$
such that the mapping 
\begin{align*}
\Lambda: \Rb \times \RN\;\; &\to\;\; \Rb \times \RN\\
(d_0,d_1)\;\; &\mapsto\;\; (\bar Q_{0}, \bar Q_{1})
\end{align*}
(where $\bar Q$ stands for $Q_{d_0,d_1}(y,s_0)$ given by \eqref{def:barQs0} and $s_0 = -\log(T- t_0)$) is linear, one to one from $\Dc_{t_0,A}$ onto $\hat \Vc_A(s_0)$ defined by \eqref{def:VcA} and maps $\partial \Dc_{t_0,A}$ into $\partial \hat \Vc_A(s_0)$. \\

\noindent $(II)$ For all $(d_0,d_1) \in \Dc_{t_0,A}$, we have $U_{d_0,d_1}(x,t_0)$ defined by \eqref{def:U0d0d1} belongs in\\
$\Sc^*(t_0, K_0, \epsilon_0, \alpha_0, A, \delta_{0,1}, \eta_0 = 0, C_{0,1}, t_0)$. More precisely,

 $(i)$ \textit{Estimates in $\Dc_1$:} $\bar Q(s_0) \in \Vc_{A, K_0}(s_0)$ with strict inequalities, except for $(\bar Q_0, \bar Q_1)(s_0)$ in the sense that
$$
|\bar Q_0(s_0)| \leq \frac{A}{s_0^2},\quad |\bar Q_{1,i}(s_0)|\leq \frac{A}{s_0^2},\quad |\bar Q_{2,ij}(s_0)| \leq \frac{\ln s_0}{s_0^2},  \quad  \forall i,j \in \{1, \cdots, N\},$$
$$|\bar Q_-(y,s_0)| \leq \frac{1}{s_0^2}(|y|^3 + 1),\quad |(\nabla_y \bar Q)_\bot (y,s_0)| \leq \frac{1}{s_0^2}(|y|^3 + 1), \quad \forall y\in \RN,$$
$$\|\bar Q_e(s_0)\|_{L^\infty(\RN)} \leq \frac{1}{\sqrt s_0},$$
where $\bar Q_m$ ($m = 0, 1, 2$), $\bar Q_-$, $\bar Q_\bot$ and $\bar Q_e$ are defined as in \eqref{decomqQ} and \eqref{decomQ2}.

$(ii)$ \textit{Estimates in $\Dc_2$:} For all $|x| \in \left[\frac{K_0}{4}\sqrt{|\ln(T-t_0)|(T-t_0)}, \epsilon_0\right]$, $\tau_0 = \tau_0(x,t_0) = \frac{t_0 - t(x)}{\theta(x)}$ and $|\xi| \leq \alpha_0 \sqrt{\ln \theta(x)}$, 
$$\left|\Uc(x,\xi, \tau_0) - \hat \Uc(\tau_0)\right| \leq \delta_{0,1}, \;\; |\nabla_\xi \Uc(x, \xi, \tau_0)| \leq \frac{C_{0,1}}{\sqrt{|\ln \theta(x)|}}, \;\; |\nabla^2_\xi \Uc(x, \xi, \tau_0)| \leq C_{0,1},$$
where $\Uc$, $\hat \Uc$, $t(x)$ and $\theta(x)$ are defined in \eqref{def:Uc}, \eqref{def:solUc}, \eqref{def:tx} and \eqref{def:thetax} respectively.
\end{proposition}
The proof of Proposition \ref{prop:PropInitialdata} is left to Appendix \ref{sec:preinti}. Let us assume that Proposition \ref{prop:PropInitialdata} holds and continue the proof of Proposition \ref{prop:UinSt}. The proof of Theorem \ref{theo:1} will follow from Proposition \ref{prop:UinSt} afterward.

\begin{proof}[\textbf{Proof of Proposition \ref{prop:UinSt}}] The proof of Proposition \ref{prop:UinSt} follows from the general ideas developed in \cite{MZdm97}. We proceed into two steps:\\
- In the first step, we reduce the problem of controlling $U(t)$ in $\Sc^*(t)$ to the control of $(Q_0, Q_1)(s)$ in $\hat \Vc_A(s)$, where $Q_0$ and $Q_1$ are the components of $Q(s)$ corresponding to the positive modes given in the decomposition \eqref{decomqQ} and $\hat \Vc_A(s)$ defined by \eqref{def:VcA}. This means that we reduce an infinite dimensional problem to a finite dimensional one.\\
- In the second step, we argue by contradiction to solve the finite dimensional problem thanks to the dynamics of $(Q_0, Q_1)(s)$ and a topological argument based on the variation of the finite dimensional parameters $(d_0, d_1)$ appearing in the definition of initial data \eqref{def:barQs0}.\\

\noindent\textit{Step 1: Reduction to a finite dimensional problem}. 

In this step, we show through \textit{a priori estimate} that the control of $U(t)$ in $\Sc^*(t)$ reduces to the control of $(Q_0, Q_1)(s)$ in $\hat{\Vc}_A(s)$ defined by \eqref{def:VcA}. This result crucially follows from a good understanding of the properties of the linear operator $\Ls + V$ of equation \eqref{equ:Q} in the \textit{blowup region} $\Dc_1$ together with advanced parabolic techniques applied to equation \eqref{equ:problem} involving a nonlinear gradient term for analysis in the intermediate and regular regions $\Dc_2$ and $\Dc_3$. In particular, we claim the following which is the heart of our contribution:
\begin{proposition}[\textbf{Control of $U(t)$ by $(Q_0,Q_1)(s)$ in $\hat\Vc_{A}(s)$}]\label{prop:redu} We can choose parameters $t_0 < T$, $K_0$, $\epsilon_0$, $\alpha_0$, $A$, $\delta_0$, $\eta_0$ and $C_0$ such that the following properties hold. Assume that $U(x,t_0)$ is given by \eqref{def:U0d0d1} with $(d_0,d_1) \in \Dc_{t_0,A}$. Assume in addition that for some $t^* \in [t_0,T)$, 
$$U(t) \in \Sc^*(t_0, K_0, \epsilon_0, \alpha_0, A, \delta_0, \eta_0, C_0, C_0', t), \quad \forall t \in [t_0,t^*],$$
and 
$$U(t^*) \in \partial\Sc^*(t_0, K_0, \epsilon_0, \alpha_0, A, \delta_0, \eta_0, C_0, C_0', t^*).$$
Then, we have\\
$(i)$ (Finite dimensional reduction) $(Q_0, Q_1)(s^*) \in \partial \hat{\mathcal{V}}_{A}(s^*)$ with $s^* = -\log(T - t^*)$.\\
$(ii)$ (Transversality) There exists $\mu_0 > 0$ such that for all $\mu \in (0, \mu_0)$, 
$$(Q_0, Q_1)(s^* + \mu) \not \in \hat{\mathcal{V}}_{A}(s^* + \mu),$$
hence,
$$ U(t^* + \mu') \not  \in \Sc^*(t_0, K_0, \epsilon_0, \alpha_0, A, \delta_0, \eta_0, C_0, C_0', t^* + \mu'), \quad \mu' = \mu'(t^*,\mu) > 0.$$
\end{proposition}

\begin{proof} The proof uses ideas of \cite{MZnon97} where the authors adapted the technique of \textit{a priori} estimates developed in \cite{BKnon94} and \cite{MZdm97} treated for equation \eqref{equ:sheGra} with $\alpha = 0$. Let us emphasize that the techniques introduced in \cite{BKnon94} and \cite{MZdm97} are not enough to handle the nonlinear gradient term appearing in equation \eqref{equ:Q}. Truly new ideas are needed to achieve the control of this term and this is one of the main novelties in this paper. The essential feature of the proof is that the given bootstrap bounds in Definition \ref{def:St} can be improved, except for the bounds on $(Q_0,Q_1)$. More precisely, the improvement of the bounds in the \textit{blowup region} $\Dc_1$ (except for $Q_0,Q_1$) is done through projecting equation \eqref{equ:Q} on the different components of $Q$ introduced in \eqref{decomqQ} and \eqref{decomQ2}. One can see that the components $Q_2$, $Q_-$, $Q_\bot$ and $Q_e$ corresponding to decreasing directions of the flow are small at $s = s_0$ and they remain small up to $s = s^*$, hence, they can not touch their boundary. In $\Dc_2$ and $\Dc_3$, we use advanced parabolic techniques applied to equation \eqref{equ:problem} involving a nonlinear gradient term in order to achieve the improvement. Therefore, only $Q_0$ and $Q_1$ may touch their boundary at $s = s^*$ and the conclusion follows. Since we would like to keep the proof of Proposition \ref{prop:UinSt} short, we leave the proof of Proposition \ref{prop:redu} to Section \ref{sec:reductofn} below.
\end{proof}

\noindent\textit{Step 2: Topological argument for the finite dimensional problem}. \label{step2:topoarg}

From Proposition \ref{prop:redu}, we claim that there exist $(d_0,d_1) \in \Dc_{t_0,A}$ such that the equation \eqref{equ:problem} with initial data \eqref{def:U0d0d1} has a solution 
$$U_{d_0,d_1}(t) \in \Sc^*(t_0, K_0, \epsilon_0, \alpha_0, A, \delta_0, \eta_0, C_0, C_0', t)\quad  \text{for all $t \in [t_0, T)$},$$
for suitable choice of the parameters. Note that the argument of the proof is not new and it is analogous as in \cite{MZdm97}. Let us gives its main ideas.

Let us consider $t_0, K_0, \epsilon_0, \alpha_0, A, \delta_0, \eta_0, C_0$ such that Propositions \ref{prop:redu} and \ref{prop:PropInitialdata} hold. From Proposition \ref{prop:PropInitialdata}, we have 
\begin{align*}
\forall (d_0,d_1) \in \Dc_{t_0,A}, \quad U_{d_0,d_1}(x,t_0) \in \Sc^*(t_0, K_0, \epsilon_0, \alpha_0, A, \delta_0, \eta_0, C_0, C_0', t_0),
\end{align*}
where $U_{d_0,d_1}(x,t_0)$ is given by \eqref{def:U0d0d1}. Note that $U_{d_0,d_1}(x,t_0) \in \mathcal{H}_{a}$ introduced in \eqref{def:Ha1a2}. Therefore, from the local existence theory for the Cauchy problem of \eqref{equ:problem} in $\mathcal{H}_{a}$, we can define for each $(d_0,d_1) \in \Dc_{t_0,A}$ a maximum time $t_*(d_0,d_1) \in [t_0,T)$ such that 
$$U_{d_0,d_1}(t) \in \Sc^*(t_0, K_0, \epsilon_0, \alpha_0, A, \delta_0, \eta_0, C_0, C_0', t) , \quad \forall t \in [t_0,t_*).$$ 
If $t_*(d_0,d_1) = T$ for some $(d_0, d_1) \in \mathcal{D}_{t_0,A}$, then the proof is complete. Otherwise, we argue by contradiction and assume that $t_*(d_0, d_1) < T$ for any $(d_0, d_1) \in \mathcal{D}_{t_0,A}$. By continuity and the definition of $t_*$, the solution $U_{d_0,d_1}(t)$ at time $t = t_*$ is on the boundary of $\Sc^*(t_*)$. From part $(i)$ of Proposition \ref{prop:redu}, we have
$$(Q_0, Q_1)(s_*) \in \partial\hat{\mathcal{V}}_A(s_*) \quad \text{with} \quad s_* = -\log(T - t_*).$$
Hence, we may define the rescaled flow $\Gamma$ at $s = s_*$ for $Q_0$ and $Q_1$ as follows:
\begin{align*}
\Gamma:\quad \mathcal{D}_{t_0,A}\quad &\mapsto \quad \partial([-1,1] \times [-1,1]^N)\\
(d_0, d_1)\quad &\to \quad \left(\frac{s_*^2}{A}Q_0(s_*), \frac{s_*^2}{A}Q_1(s_*)\right).
\end{align*}
It follows from part $(ii)$ of Proposition \ref{prop:redu} that $\Gamma$ is continuous. If we manage to prove that the degree of $\Gamma$ on the boundary is different from zero, then we have a contradiction from the degree theory. Let us prove that. From part $(i)$ Proposition \ref{prop:PropInitialdata}, we see that if $(d_0,d_1) \in \partial \mathcal{D}_{t_0,A}$, then 
$$Q(s_0) \in \mathcal{V}_{A,K_0}(s_0) \quad \text{and} \quad (Q_0, Q_1)(s_0) \in \partial \hat{\mathcal{V}}_{A}(s_0).$$
Using part $(ii)$ of Proposition \ref{prop:redu}, we see that $Q(s)$ must leave $\mathcal{V}_{A,K_0}(s)$ at $s = s_0$, hence, $s_*(d_0,d_1) = s_0$. Using again part $(i)$ of Proposition \ref{prop:PropInitialdata}, we see that the degree of $\Gamma$ on the boundary must be different from zero. This gives us a contradiction (by the index theory) and concludes the proof of Proposition \ref{prop:UinSt}, assuming that Propositions \ref{prop:redu} and \ref{prop:PropInitialdata} hold.
\end{proof}

Let us now give the proof of Theorem \ref{theo:1} from Proposition \ref{prop:UinSt}, assuming that Propositions \ref{prop:PropInitialdata} and \ref{prop:redu} hold. 

\begin{proof}[\textbf{Proof of Theorem \ref{theo:1}, assuming Propositions  \ref{prop:PropInitialdata} and \ref{prop:redu}}] We give in this part the proof of Theorem \ref{theo:1}. We will present the proofs of item $(i)$, $(ii)$ and $(iii)$ separately.

$(i)$ The proof of part $(i)$ is equivalent to the proof of \eqref{eq:goalQ} through the change of variables \eqref{def:Qpsi}. From Proposition \ref{prop:UinSt}, we know that equation \eqref{equ:problem} with the initial data given by \eqref{def:U0d0d1} has the solution $U(t) \in \Sc^*(t)$ for all $t \in [t_0, T)$. From part $(i)$ of Definition \ref{def:St}, we have $Q(s) \in \Vc_{A, K_0}(s)$ for all $s \in [-\log(T - t_0), +\infty)$, where $A$ and $K_0$ are some fixed large constants. In Proposition \ref{prop:propSt} below, we show that if  $Q(s) \in \Vc_{A, K_0}(s)$, then 
\begin{equation}\label{est:Qs}
\|Q(s)\|_{W^{1,\infty}(\RN)} \leq \frac{C(A,K_0)}{\sqrt{s}},
\end{equation}
which is the conclusion of \eqref{eq:goalQ} as well as \eqref{equ:behTh1}.\\

$(ii)$ Note that part $(i)$ also implies that $U$ and $\nabla U$ blow up in finite time $T$ at the origin. Indeed, by the definition \eqref{def:Phi_alpha} of $\Phi_\alpha$ and part $(i)$ of Theorem \ref{theo:1}, we have
$$U(0,t) \sim -\ln(T-t) \to +\infty \quad \text{as}\quad t \to T,$$
hence, $U$ and $e^U$ blow up in finite time $T$ at the origin.  As for $\nabla U$, we write from \eqref{def:Qpsi}, 
$$\nabla W - \nabla \Phi_\alpha = \frac{\nabla Q - \nabla \Phi_\alpha Q}{Q + \psi_\alpha}.$$
From the definition \eqref{def:Qpsi} of $\psi_\alpha$ and \eqref{est:Qs}, we see that 
$$\forall |y| \leq K_0\sqrt s,\quad  Q(y,s) + \psi_\alpha(y,s) \geq \psi_\alpha(K_0\sqrt s, s) - \|Q(s)\|_{L^\infty} > \frac{C}{K_0^2}.$$
We also show in Proposition \ref{prop:propSt} that if $Q(s) \in \Vc_{A, K_0}(s)$, then 
$$\forall y \in \RN, \quad |Q(y,s)| + |\nabla Q(y,s)| \leq \frac{C(A, K_0)\ln s}{s^2}(|y|^3 + 1).$$
Note from the definition \eqref{def:Phi_alpha} of $\Phi_\alpha$ that $|\nabla \Phi_\alpha(y/\sqrt s)| \leq \frac{C|y|}{s}$, we have for all $|y| \leq K_0 \sqrt s$, 
$$\left|\nabla W(y,s) - \nabla \Phi_\alpha\left(\frac {y}{\sqrt s}\right)\right| \leq \frac{C(A, K_0)\ln s}{s^2}(|y|^3 + 1) + \frac{C(A, K_0)\ln s}{s^3}(|y|^4 + 1). $$
From \eqref{def:changeVar} and \eqref{def:Phi_alpha}, we write
\begin{align*}
&\left|\nabla W(y,s) - \nabla \Phi_\alpha\left(\frac {y}{\sqrt s}\right)\right| = \left|\sqrt {T-t} \,\nabla U(x,t) + \frac{ye^{\Phi_\alpha\left(\frac{y}{\sqrt {|\ln (T-t)|}}\right)}}{(2 + 2\alpha) |\ln (T-t)|}\right|\\
&\leq \frac{C(A, K_0)\ln |\ln (T-t)|}{|\ln (T-t)|^2}(|y|^3 + 1) + \frac{C(A, K_0) \ln |\ln (T-t)|}{|\ln (T-t)|^3}(|y|^4 + 1).
\end{align*}
Put $y = y(s) = s^\frac{1}{4}\omega = |\ln (T-t)|^\frac 14 \omega$ with $|\omega| = 1$, we see that the right hand side of the above estimate is bounded by $\frac{\ln |\ln (T-t)|}{|\ln (T-t)|^\frac 54} \ll \frac{1}{|\ln (T-t)|^\frac 34}$ as $t \to T$. Therefore, we have 
$$\nabla U(|\ln (T-t)|^\frac 14 \sqrt {T-t},t) \sim \frac{C}{|\ln (T-t)|^\frac 34 \sqrt{T-t}} \quad \text{as} \quad t \to T.$$
Since $|\ln (T-t)|^\frac 34 \sqrt{T-t} \to 0$ as $t \to T$, hence, $\nabla U$ blows up at time $T$ at the origin. 

In order to prove that $e^U$ and $\nabla U$ blow up only at the origin, we use the following result:

\begin{proposition}[\textbf{No blowup under some threshold}] \label{prop:noblowup} Let $u(\xi, \tau)$ satisfy the following inequality: for all $K > 0$, for all $|\xi| < 1$ and $\tau \in [0, 1)$,
\begin{equation}\label{ine:pb}
\left\{\begin{array}{rl}
|\partial_\tau u - \Delta u|  &\leq K( 1+ e^u + |\nabla u|^2),\\
|\nabla(\partial_\tau u - \Delta u)|  &\leq K|\nabla(e^u + |\nabla u|^2)|.
\end{array}\right.
\end{equation}
Assume that there is a constant $\epsilon = \epsilon(K, N) > 0$ small enough such that 
\begin{equation}\label{con:ineUnU}
(1 - \tau)e^{u(\xi, \tau)} + \sqrt{1 - \tau}|\nabla u(\xi, \tau)| \leq \epsilon, \quad \forall |\xi| < 1, \;\tau \in [0, 1),
\end{equation}
then, 
\begin{equation}\label{est:UnU}
e^{u(\xi, \tau)} + |\nabla u(\xi, \tau)| \leq C\epsilon, \quad \forall |\xi|\leq \frac{1}{8}, \; \tau \in [0, 1).
\end{equation}
In particular, $e^u$ and $\nabla u$ do not blow up at $\xi = 0$ and $\tau = 1$. 
\end{proposition}
\begin{proof} The proof of this result uses ideas given in Giga and Kohn \cite{GKcpam89}, where \eqref{ine:pb} is considered without the gradient term and the nonlinear source term $e^u$ replaced by $|u|^{p-1}u$. The proof in \cite{GKcpam89} uses a truncation technique together with the smoothness effect of the heat semigroup $e^{\tau\Delta}$ and some type of Gronwall's argument. Although some advanced parabolic regularities are needed to treat our problem involving the nonlinear gradient term, but the same argument to those of \cite{GKcpam89} can be extended to our case without difficulties. Since the proof is long and technical, we give the proof in Appendix \ref{sec:noblowup}.
\end{proof}

Let us apply Proposition \ref{prop:noblowup} to $\Uc(x_0, \xi, \tau)$, where $\Uc(x_0, \xi, \tau)$ and $x_0$ are defined as in \eqref{def:Uc} and  \eqref{def:tx}. Recall from \eqref{eq:Ucxt} that $\Uc$ solves the following equations:
\begin{align*}
\partial_\tau \Uc &= \Delta_\xi \Uc + \alpha|\nabla_\xi \Uc|^2 + e^{\Uc},
\end{align*}
hence, \eqref{ine:pb} is satisfied. For the condition \eqref{con:ineUnU}, let us write from the definition  \eqref{def:Uc} of $\Uc(x_0, \xi, \tau)$, the definition \eqref{def:tx} of $x_0$ and part $(i)$ of Theorem \ref{theo:1},
\begin{align*}
&\sup_{|\xi| \leq 1, \tau \in [0,1)}(1 - \tau)e^{\Uc(x_0, \xi, \tau)}\\
&\quad  = \sup_{|x - x_0| \leq \sqrt{T - t(x_0)}, \tau \in [0,1)}(1 - \tau)^2(T - t(x_0))e^{U(x, t(x_0) + \tau(T-t(x_0)))}\\
& \qquad \leq \sup_{|x - x_0| \leq \frac{|x_0|}{2}}(T - t(x_0))e^{U(x, t)} \leq e^{\Phi_\alpha\left(\frac{|x_0|/2}{\sqrt{(T-t(x_0))|\log(T-t(x_0))|}}\right)} + \frac{C}{\sqrt{|\log(T-t(x_0))|}}\\
&\qquad \quad \leq e^{\Phi_{\alpha}(K_0/8)} + \frac{C}{\sqrt{|\ln (T-t(x_0))|}} \leq \epsilon_0(K_0, |x - x_0|) \to 0,
\end{align*} 
as $K_0 \to +\infty$ and $|x - x_0| \to 0$. Similarly, we have 
\begin{align*}
&\sup_{|\xi| \leq 1, \tau \in [0,1)}\sqrt{1 - \tau}|\nabla_\xi \Uc(x_0, \xi, \tau)|\\
&\quad \leq \sup_{x \in \RN, \tau \in [0,1)} \sqrt{T - t(x_0)}|\nabla_x U(x,t(x_0) + \tau(T - t(x_0)))|\\
&\qquad \leq \frac{C}{\sqrt{|\log(T -t(x_0))|}} \left(\|\zeta e^{\Phi_\alpha(\zeta)}\|_{L^\infty(\RN)} + 1 \right)\\
&\qquad \quad \leq \epsilon_0'(|x - x_0|) \to 0 \quad \text{as $|x - x_0| \to 0$}.
\end{align*}
which verifies the condition \eqref{con:ineUnU} for $\Uc(x_0, \xi, \tau)$ and $\nabla_\xi(x_0, \xi, \tau)$. Hence, we can apply Proposition \ref{prop:noblowup} to $\Uc(x_0, \xi, \tau)$ to deduce that $\xi = 0$ is not a blowup point of $e^{\Uc(x_0, \xi, \tau)}$ and $\nabla_\xi \Uc(x_0, \xi, \tau)$, which means that $x_0 \ne 0$ is not a blowup point of $e^U$ and $\nabla U$. Let us insist on the fact that our argument works for any $x_0 \ne 0$ without any smallness assumptions, thanks to the adapted definition of $t(x_0)$ given in \eqref{def:tx}. This proves the single point blowup result for $e^U$ and $\nabla U$, and concludes the proof of part $(ii)$ of Theorem \ref{theo:1}.\\

$(iii)$ We divide the proof into two steps. We first show the existence of the final profile $U^*_\alpha$, then we find an equivalent of $U^*_\alpha$ which concludes the proof of part $(iii)$ of Theorem \ref{theo:1}. We claim the following:
\begin{proposition}[\textbf{Existence of the final blowup profile}] \label{prop:estprofile} Let $U(t)$ be a solution of equation \eqref{equ:problem} which blows up in finite time $T$ at the origin and verifies the asymptotic behavior \eqref{equ:behTh1}. There exists a function $U^*_\alpha(x) \in \mathcal{C}^1(\RN \setminus \{0\})$ such that $U(x,t) \to U^*_\alpha(x)$ and $\nabla U(x,t) \to \nabla U^*_\alpha(x)$ as $t \to T$, uniformly on compact sets of $\RN \setminus \{0\}$.
\end{proposition}
\begin{proof} The proof uses the same argument given by Merle \cite{Mercpam92} treated for equation \eqref{equ:sheGra} with $\alpha = 0$, which relies on some classical regularity argument of parabolic problem. In comparison with the work in \cite{Mercpam92}, the only difference is that we need to extend the \textit{no blowup under some threshold} result of Giga and Kohn \cite{GKcpam89} to our equation \eqref{equ:problem}, which is Proposition \ref{prop:noblowup}. Let us denote by $H = \nabla U$, we write from equation \eqref{equ:problem},
\begin{align*}
\partial_t U &= \Delta U + \alpha |\nabla U|^2 + e^U,\\
\partial_t H &= \Delta H + \alpha \nabla (|H|^2) +  He^U .
\end{align*}
From Proposition \ref{prop:noblowup}, we proved in part $(ii)$ that $e^U$ and $\nabla U$ are uniformly bounded on $\Omega \times [0, T)$ for any compact set $\Omega \subset \RN \setminus \{0\}$. By parabolic regularity techniques, similar to Proposition \ref{prop:noblowup}, we can show that $\partial_t U$ and $\partial_t H$ are also bounded on $\Omega' \times [T/2, T)$ for any $\Omega' \subset \RN \setminus \{0\}$. Therefore, as in \cite{Mercpam92} (see Proposition 2.2, page 269), we conclude that there exists $U^*_\alpha$ in $\mathcal{C}^1(\RN \setminus \{0\})$ such that $U(x, t) \to U^*_\alpha(x)$ and $\nabla U(x,t) \to \nabla U^*_\alpha(x)$ as $t \to T$, uniformly on each compact set of $\RN \setminus \{0\}$. This concludes the proof of Proposition \ref{prop:estprofile}.
\end{proof}

Let us now find an equivalence of $U^*_\alpha$ and $\nabla U^*_\alpha$ in order to complete the proof of part $(iii)$ of Theorem \ref{theo:1}. To this end, we use the same argument as in the proof of Proposition \ref{prop:estprofile} to show that the limit of $\Uc (x, 0, \tau)$ and $\nabla_\xi \Uc(x, 0, \tau)$ as $\tau \to 1$ exist for $x$ sufficiently small. Moreover, using part $(ii)$ of Definition \ref{def:St}, we see that 
$$\lim_{\tau \to 1} \Uc (x, 0, \tau) \sim \hat{\Uc }_{K_0}(1) = - \ln\left(\frac{K_0^2}{4 + 4\alpha}\right),$$
and 
$$\lim_{\tau \to 1}| \nabla_\xi \Uc(x, 0, \tau)| \leq \frac{C}{\sqrt{|\ln (T - t(x))|}},$$
for $|x|$ and $1/K_0$ sufficiently small. 

From Proposition \ref{prop:estprofile}, we have that $\lim_{t \to T}U(x, t) = U^*_\alpha(x)$ and $\lim_{t \to T} \nabla U(x, t) = \nabla U^*_\alpha(x)$. Hence, from the definitions \eqref{def:Uc} and \eqref{def:tx}, we derive 
\begin{align*}
U^*_\alpha(x) = \lim_{t \to T}U(x, t) &= \lim_{\tau \to 1} \big[-\ln (T - t(x)) + \Uc (x, 0, \tau)\big]\\
& \sim -\ln (T - t(x)) - \ln\left(\frac{K_0^2/16}{4 + 4\alpha}\right)\\
& \sim -\ln |x|^2 + \ln|\ln |x|^2| + \ln (K_0^2/16) - \ln\left(\frac{K_0^2/16}{4 + 4\alpha}\right)\\
& = \ln \left(\frac{(8 + 8\alpha)|\ln|x||}{|x|^2} \right)  \quad \text{as} \quad |x| \to 0,
\end{align*}
and 
\begin{align*}
|\nabla_x U^*_\alpha(x)| = |\lim_{t \to T} \nabla_x U(x, t)| &= \left|\lim_{\tau \to 1}\frac{\nabla_\xi \Uc(x, 0, \tau)}{\sqrt{T - t(x)}}\right|\\
&\leq \frac{C}{\sqrt{|\ln (T - t(x))|(T - t(x))}} \leq \frac{C}{|x|}  \quad \text{as} \quad |x| \to 0.
\end{align*}
This concludes the proof of Theorem \ref{theo:1} assuming that Propositions \ref{prop:PropInitialdata}, \ref{prop:redu} and \ref{prop:noblowup} hold. 
\end{proof}

\section{Reduction to a finite dimensional problem.}\label{sec:reductofn}
This section is the heart of our analysis.  We aim at proving Proposition \ref{prop:redu} which reduces the problem to a finite dimensional one. We proceed in two parts. In the first part, we derive \textit{a priori estimates} on $U(t)$ in $\Sc^*(t)$. In the second part, we show that these new bounds are better than those defined in $\Sc^*(t)$, except for the bounds on the components $Q_0(s)$ and $Q_1(s)$. This means that the problem is reduced to the control of a finite dimensional function $(Q_0,Q_1)(s)$ which is the conclusion $(i)$ of Proposition \ref{prop:redu}. The outgoing transversality property stated in part $(ii)$ of Proposition \ref{prop:redu} is a direct consequence of the dynamics of the modes $Q_0$ and $Q_1$. Let us start with the first part.

\subsection{A priori estimates.}\label{sec:apri}
We have the following estimates:

\begin{proposition}[\textbf{A priori estimate in $\Dc_1$}]\label{prop:D1} There exist $K_{0,2} > 0$ and $A_{0,2} > 0$ such that for all $K_0 \geq K_{0,2}$, $\epsilon_0 > 0$, $A \geq A_{0,2}$, $\lambda^* > 0$, $C_{0,2} > 0$,  there exists $t_{0,2}(K_0, \epsilon_0, A, \lambda^*, C_{0,2})$ with the following property: For all $\delta_0 \leq \frac{1}{2}|\hat \Uc(1)|$, $\alpha_0 > 0$, $C_0 > 0$ and $\eta_0 \leq \eta_{0,2}$ for some $\eta_{0,2}(\epsilon_0) > 0$, $\lambda \in [0, \lambda^*]$ and $t_0 \in [t_{0,2}, T)$, assume that 
\begin{itemize}
\item $U(x, t_0)$ is given by \eqref{def:U0d0d1} and $(d_0,d_1)$ is chosen such that $(\bar Q_0(s_0), \bar Q_1(s_0)) \in \hat \Vc_{A}(s_0)$ where $s_0 = \ln (T-t_0)$ and $\hat \Vc_A$ is defined in \eqref{def:VcA}.
\item for some $\sigma \geq s_0$, we have for all $t \in [T - e^{-\sigma}, T - e^{-(\sigma + \lambda)}]$, $$U(x,t) \in \Sc^*(t_0, K_0, \epsilon_0, \alpha_0, A, \delta_0, C_0, C_0', \eta_0, t).$$
\end{itemize}
Then, we have for all $s \in [\sigma, \sigma+\lambda]$,
\begin{itemize}
\item[(i)] (ODEs satisfied by the positive and null modes of $Q$) 
\begin{equation}\label{equ:odev0}
\left|Q_0'(s) - Q_0(s) \right| \leq \frac{C}{s^2},
\end{equation}
\begin{equation}\label{equ:odev01}
\forall i \in \{1, \cdots, N\}, \quad\left|Q_{1,i}'(s) - \frac{1}{2}Q_{1,i}(s) \right| \leq \frac{C}{s^2},
\end{equation}
and 
\begin{equation}\label{equ:odev2}
\forall i, j\in \{1, \cdots, N\}, \quad \left|Q_{2, ij}'(s) + \frac{2}{s}Q_{2,ij}(s) \right| \leq \frac{CA}{s^3}.
\end{equation}
\item[(ii)] (Control of the negative and outer part of $Q$)\\
- For $\sigma \geq s_0$:
\begin{equation}
\left|Q_-(y,s) \right| \leq C\Big(Ae^{-\frac{s- \sigma}{2}} +  A^2e^{-(s - \sigma)^2} + (s - \sigma)\Big)s^{-2}(1 + |y|^3),
\end{equation}
\begin{equation}
|Q_e(y,s)| \leq C\Big(A^2e^{-\frac{s - \sigma}{2}} + AK_0^3 e^{s - \sigma} + K_0^3(s - \sigma + 1)\Big)s^{-1/2}.
\end{equation}
- For $\sigma = s_0$:
\begin{equation}
|Q_-(y,s)| \leq C(1 + s - \sigma)s^{-2}(1 + |y|^3), \quad |Q_e(y,s)| \leq CK_0^3(1 + s - \sigma)e^{s-\sigma}s^{-1/2}.
\end{equation}
\item[(iii)] (Control of the gradient of $Q$)\\
- For $\sigma \geq s_0$:
\begin{equation}
|(\nabla Q)_\bot(y,s)| \leq C\Big(Ae^{-\frac{s- \sigma}{2}} +  C(K_0)C_0e^{-(s - \sigma)^2} + \sqrt{s  -\sigma}+ (s - \sigma)\Big)s^{-2}(1 + |y|^3).
\end{equation}
- For $\sigma = s_0$:
\begin{equation}
|(\nabla Q)_\bot(y,s)| \leq C(1 + \sqrt{s - \sigma} + s- \sigma)s^{-2}(1 + |y|^3).
\end{equation}
\end{itemize}
\end{proposition}
\begin{proof} The proof of this proposition is completely the same as in \cite{MZnon97} because our equation \eqref{equ:Q} and the shrinking set $\Vc_{A, K_0}$ defined in part $(i)$ of Definition \ref{def:St} are analogous as those defined in that paper. Note that the coefficient $(\alpha - 1)$ apearing in \eqref{def:GQ} and the nonlinear term $Q^2$ in equation \eqref{equ:Q} are replaced by some constants $a > 1$ and $Q^p$ with $p > 1$ in \cite{MZnon97}, and that these changes don't affect to their analysis. For this reason, we kindly refer the reader to Lemma 3.2 at page 1523 in \cite{MZnon97} for a similar statement and Appendix B for all details of its proof.
\end{proof}

We now turn to the \textit{a priori estimates} of $U$ in $\Dc_2$. We prove the following:
\begin{proposition}[\textbf{A priori estimate in $\Dc_2$}]\label{prop:D2}
There exists $K_{0,3} > 0$ such that for all $K_0 \geq K_{0,3}$, $\delta_1 \leq 1$, $\xi_0 \geq 1$ and $C_{0,1}^* > 0$, $C_{0,2}^*$, $C_{0,3}^*$, we have the following property: Assume that $\Uc$ is a solution of equation 
\begin{equation}\label{eq:UcD2}
\partial_\tau \Uc = \Delta \Uc + \alpha|\nabla \Uc|^2 + e^\Uc,
\end{equation}
for $\tau \in [\tau_1, \tau_2]$ with $0 \leq \tau_1 \leq \tau_2 \leq 1$. Assume in addition, for all $\tau \in [\tau_1, \tau_2]$,
\begin{itemize}
\item[(i)] for all $|\xi| \leq 2\xi_0$, $|\Uc(\xi, \tau_1) - \hat \Uc(\tau_1)| \leq \delta_1$ and $|\nabla \Uc(\xi, \tau_1)| \leq \frac{C_{0,1}^*}{\xi_0}$, where $\hat \Uc(\tau)$ is given by \eqref{def:solUc},
\item[(ii)] for all $|\xi| \leq \frac 74\xi_0$, $|\nabla \Uc(\xi, \tau)| \leq \frac{C_{0,2}^*}{\xi_0}$ and $|\nabla^2 \Uc(\xi, \tau)| \leq C_{0,3}^*$,
\item[(iii)] for all $|\xi|\leq  \frac 74\xi_0$, $\Uc(\xi, \tau) \leq \frac{1}{2}\hat \Uc(\tau)$.
\end{itemize}
Then, for $\xi_0 \geq \xi_{0,3}(C_0^*)$, there exists $\epsilon = \epsilon(K_0, C_{0,2}^*, C_{0,3}^*, \delta_1, \xi_0)$ such that for all $|\xi| \leq \xi_0$ and $\tau \in [\tau_1, \tau_2]$,
$$|\Uc(\xi, \tau) - \hat \Uc(\tau)| \leq \epsilon, \quad |\nabla \Uc(\xi, \tau)| \leq \frac{2C_{0,1}^*}{\xi_0},$$
where $\epsilon \to 0$ as $(\delta_1, \xi_0) \to (0, +\infty)$.
\end{proposition}
\begin{proof} We first deal with the gradient estimate. We aim at proving that under the hypothesis of Proposition \ref{prop:D2}, we have 
\begin{equation}\label{est:nabUcD2}
\forall |\xi| \leq \frac{5}{4}\xi_0, \;\; \tau \in [\tau_1, \tau_2], \quad |\nabla \Uc(\xi, \tau)| \leq \frac{2C_{0,1}^*}{\xi_0},
\end{equation}
provided that $\xi_0 \geq \xi_{0,3}(C_0^*)$. To do so, let us denote by $\theta = |\nabla \Uc|^2$ and write from \eqref{eq:UcD2}, 
$$\partial_\tau \theta \leq \Delta \theta + C(C_{0,3}^*)\theta, \quad \forall |\xi| \leq \frac{7}{4}\xi_0, \; \tau \in [\tau_1, \tau_2],$$
where we used the fact that $2 \nabla \Uc \cdot  \nabla(\Delta\Uc) \leq \Delta \theta$ and $2\nabla \Uc \cdot \nabla (\alpha|\nabla \Uc|^2 + e^\Uc) = (4\alpha \Delta \Uc + 2e^\Uc)\theta \leq C(C_{0,3}^*)\theta$. 

Let us consider $\rho_1 \in \Cc^\infty(\RN)$ such that $\rho_1 \in [0,1]$, $\rho_1(\xi) = 1$ for $|\xi| \leq \frac{3}{2}\xi_0$ and $\rho_1(\xi)= 0$ for $|\xi| \geq \frac{7}{4}\xi_0$, $|\nabla \rho_1| \leq \frac{C}{\xi_0}$ and $|\Delta \rho_1| \leq \frac{C}{\xi_0^2}$. Then, $\theta_1 = \rho_1 \theta$ satisfies the inequality:
\begin{align*}
\partial_\tau \theta_1 &\leq \Delta \theta_1 - 2\nabla \rho_1 \cdot \nabla \theta - \Delta \rho_1 \theta + C(C_{0,3}^*)\theta_1\\
&\leq \Delta \theta_1 + C(C_{0,2}^*, C_{0,3}^*)\xi_0^{-2}\mathbf{1}_{\{\frac{3}{2}\xi_0 \leq |\xi| \leq 2\xi_0\}} + C(C_{0,3}^*)\theta_1.
\end{align*}
Let $\theta_2 = e^{-C(C_{0,3}^*)\tau}\theta_1$, we have 
$$\partial_\tau \theta_2 \leq \Delta \theta_2 + C(C_{0,2}^*, C_{0,3}^*)\xi_0^{-2}\mathbf{1}_{\{\frac{3}{2}\xi_0 \leq |\xi| \leq 2\xi_0\}} \quad \text{and} \quad 0 \leq \theta_2(\tau_1) \leq \frac{{C_{0,1}^*}^2}{\xi_0^2}.$$
Therefore, by the maximum principle, we deduce
$$\forall |\xi| \leq \frac{5}{4}\xi_0, \; \tau \in [\tau_1, \tau_2], \quad \theta(\xi,\tau) \leq \frac{{C_{0,1}^*}^2}{\xi_0^2} + C(C_{0,2}^*, C_{0,3}^*)^2\xi_0^{-2}e^{-C'\xi_0^2},$$
which  yields the conclusion \eqref{est:nabUcD2}.\\

We now turn to the estimate on $\Uc$. We use here the same argument as in \cite{MZnon97}, and consider $\Uc_1$ a solution to \eqref{eq:UcD2} such that for all $|\xi| \leq 2$ and $\tau \in [\tau_1, \tau_2]$, $$|\Uc_1(\xi, \tau_1) - \hat \Uc(\tau_1)| \leq \delta_1, \quad |\nabla \Uc_1(\xi, \tau)| \leq \epsilon.$$
Let us show that for $|\xi| \leq 2$ and $\tau \in [\tau_1, \tau_2]$, 
$$|\Uc_1(\xi, \tau) - \hat \Uc(\tau)| \leq C(K_0)\epsilon + \delta_1.$$
We write for all $\tau \in [\tau_1, \tau_2]$,
\begin{align*}
\Uc_1(0,\tau) &= \frac{1}{|B_2(0)|}\int_{|\xi| \leq 2}\Uc_1(\xi, \tau)d\xi + \Uc_2(\tau),\\
e^{\Uc_1(0, \tau)} &= \frac{1}{|B_2(0)|}\int_{|\xi| \leq 2}e^{\Uc_1(\xi, \tau)}d\xi + \Uc_3(\tau),
\end{align*}
where $|B_2(0)|$ is the volume of the sphere of radius 2 in $\RN$, $\|\Uc_2\|_{L^\infty} \leq 2\epsilon$ and $\|\Uc_3\|_{L^\infty} \leq C\epsilon.$

For $\epsilon$ small enough, if we consider in the distribution sense, 
$$\tilde \Uc(\tau) = \frac{1}{|B_2(0)|}\int_{|\xi| \leq 2} \Uc_1(\xi, \tau)d\xi,$$
then, we have from \eqref{eq:UcD2}, 
$$e^{\tilde \Uc} - C\epsilon - |\alpha| \epsilon^2 \leq \frac{d \tilde{\Uc}}{d\tau} \leq e^{\tilde \Uc} + C\epsilon + |\alpha| \epsilon^2,$$
and 
$$|\tilde{\Uc}(\tau_1) - \hat \Uc(\tau_1)| \leq C\epsilon + \delta_1.$$
From \eqref{eq:UcD2}, we obtain by classical \textit{a priori} estimates that for all $\tau \in [\tau_1, \tau_2]$, $|\tilde{\Uc}(\tau) - \hat \Uc(\tau)| \leq C(K_0)\epsilon + \delta_1$ (since $C_1(K_0) \leq |\hat \Uc(\tau)| \leq C_1'(K_0)$). Therefore, for $|\xi| \leq 2$ and $\tau \in [\tau_1, \tau_2]$, we have 
$|\Uc_1(\xi, \tau) - \hat \Uc(\tau)| \leq C(K_0)\epsilon + \delta_1$. Applying this result to $\Uc_1(\xi, \tau) = \Uc(\xi - \bar \xi_0, \tau)$ with $|\bar \xi_0| \in [-\xi_0 + 2, \xi_0 + 2]$ and using the hypothesis, we obtain
$$\forall |\xi| \leq \xi_0, \; \tau \in [\tau_1, \tau_2], \quad |\Uc(\xi, \tau) - \hat \Uc(\tau)| \leq \epsilon,$$
where $\epsilon = \epsilon(\delta_1, \xi_0) \to 0$ as $(\delta_1, \xi_0) \to (0, +\infty)$. This concludes the proof of Proposition \ref{prop:D2}.
\end{proof}

For the \textit{a priori estimates} of $U$ in $\Dc_3$, we claim the following:
\begin{proposition}[\textbf{A priori estimate in $\Dc_3$}]\label{prop:D3}
For all $\epsilon > 0$, $\epsilon_0 > 0$, $\sigma_0 > 0$, there exists $t_{0,4}(\epsilon, \epsilon_0, \sigma_0) < T$ such that for all $t_0 \in [t_{0,4}, T)$, if $U$ is a solution of \eqref{equ:problem} on $[t_0,t_*]$ for some $t_* \in [t_0, T)$ satisfying 
\begin{itemize}
\item[(i)] for all $|x| \in \left[\frac{\epsilon_0}6, \frac{\epsilon_0}{4}\right]$ and $t \in [t_0, t_*]$,
\begin{equation}
i = 0, 1, 2, \quad |\nabla^i U(x,t)| \leq \sigma_0,
\end{equation}
\item[(ii)] For $|x| \geq \frac{\epsilon_0}{6}$, $U(x,t_0) = \hat U^*(x)$ where $\hat U^*$ is defined in \eqref{def:Uhatstar},
\end{itemize}
then for all $|x| \in \left[\frac{\epsilon_0}{4}, +\infty\right)$ and $t \in [t_0, t_*]$, 
\begin{equation}
|U(x,t) - U(x,t_0)| + |\nabla U(x,t) - \nabla U(x,t_0)| \leq \epsilon.
\end{equation}
\end{proposition}
\begin{proof} We only deal with the estimate on $U$ for $|x| \geq \frac{\epsilon_0}{4}$ because the estimate on $\nabla U$ can be obtained similarly. We argue by contradiction. Let us consider $t_\epsilon \in (t_0, t_*)$ such that for all $t \in [t_0, t_\epsilon)$, 
\begin{align}
&\|U(x,t) - U(x,t_0)\|_{L^\infty(|x| \geq \frac{\epsilon_0}{4})} \leq \epsilon, \quad \|\nabla U(x,t) - \nabla U(x,t_0)\|_{L^\infty(|x| \geq \frac{\epsilon_0}{4})} \leq \epsilon,\label{eq:asupD3}\\
&\|U(x,t_\epsilon) - U(x,t_0)\|_{L^\infty(|x| \geq \frac{\epsilon_0}{4})} = \epsilon \quad \|\nabla U(x,t_\epsilon) - \nabla U(x,t_0)\|_{L^\infty(|x| \geq \frac{\epsilon_0}{4})} = \epsilon.\label{eq:asupD31}
\end{align}
We remark from \eqref{def:Uhatstar} that 
$$e^{U(x,t_0)} = e^{\hat U^*(x)} \leq C(\epsilon_0), \quad  |\nabla U(x, t_0)|^2 = |\nabla \hat U^*(x)|^2 \leq C(\epsilon_0), \quad \forall |x| \geq \frac{\epsilon_0}{6}.$$
Hence, from \eqref{eq:asupD3}, we have 
$$|F(U, \nabla U)| = |e^{U(x,t)} + \alpha |\nabla U(x,t)|^2| \leq C(\epsilon_0), \quad \forall |x| \geq \frac{\epsilon_0}{6}, \; t \in [t_0, t_\epsilon).$$
Consider $U = \psi + U_1$, where $\psi(x) = -\ln(1 + a|x|^2)$ is introduced in \eqref{def:Ha1a2}, we write from \eqref{equ:problem},
$$\partial_t U_1 = \Delta U_1 + F(U, \nabla U) + \Delta \psi. $$
From assumption $(i)$, we have in fact for all $t \in [t_0, t_\epsilon]$ and $|x| \in [\frac{\epsilon_0}{6}, \frac{\epsilon_0}{4}]$, $|F(U, \nabla U)| \leq C(\sigma_0)$. We then consider $U_2(x,t) = \rho_2(x)U_1(x,t)$, where $\rho_2(x) \in \Cc^\infty(\RN)$, $\rho_2(x) = 1$ for $|x| \geq \frac{\epsilon_0}{5}$, $\rho_2(x) = 0$ for $|x| \leq \frac{\epsilon_0}{6}$, $|\nabla \rho_2| \leq \frac{C}{\epsilon_0}$ and $|\Delta \rho_2| \leq \frac{C}{\epsilon_0^2}$. We then have 
$$\partial_t U_2 = \Delta U_2 - 2\nabla \rho_2 \cdot \nabla U_1 - \Delta \rho_2 U_1 + \rho_2 F(U,\nabla U) + \rho_2\Delta \psi.$$
Note from the definitions of $\rho_2$ and $\psi$ and assumption $(i)$ that for all $t \in [t_0, t_*]$, 
$$|2\nabla \rho_2 \cdot \nabla U_1| + |\Delta \rho_2 U_1| \leq C(\epsilon_0, \sigma_0)\mathbf{1}_{\{\frac{\epsilon_0}{6}\leq |x| \leq \frac{\epsilon_0}{5}\}}.$$
We write 
$$\partial_t U_2 = \Delta U_2 + \tilde{f}_1(x,t) + \rho_2  F(U,\nabla U) + \rho_2 \Delta \psi,$$
with $|\tilde{f}_1(x,t)| \leq C(\epsilon_0, \sigma_0)\mathbf{1}_{\{\frac{\epsilon_0}{6}\leq |x| \leq \frac{\epsilon_0}{5}\}}.$

We denote by $S(\cdot)$ the linear heat flow, we write for all $ t \in [t_0, t_\epsilon)$, 
$$U_2(t) - S(t - t_0)U_2(t_0) = \int_{t_0}^t S(t - s)\left[\tilde{f}_1 + \rho_2 F(U,\nabla U) + \rho_2 \Delta \psi\right] ds.$$
Hence,
\begin{align*}
\|U_2(t) - U_2(t_0)\|_{L^\infty(\RN)} &\leq \|U_2(t) - S(t - t_0)U_2(t_0)\|_{L^\infty(\RN)}\\
&\quad  + \|S(t - t_0)U_2(t_0) - U_2(t_0)\|_{L^\infty(\RN)}\\
&\leq \int_{t_0}^t \left\|S(t - s)\left[\tilde{f}_1 + \rho_2 F(U,\nabla U) + \rho_2\Delta \psi \right]\right\|_{L^\infty(\RN)}ds\\
&\quad + \|S(t - t_0)\rho_2U_1(t_0) - \rho_2U_1(t_0)\|_{L^\infty(\RN)}\\
&\leq C(\epsilon, \sigma_0)(t-t_0) + \|S(t - t_0)\rho_2U_1(t_0) - \rho_2U_1(t_0)\|_{L^\infty(\RN)}.
\end{align*}
From assumption $(ii)$ and the definition \eqref{def:Uhatstar} of $\hat U^*$, we see that 
$$\|\rho_2U_1(t_0)\|_{L^\infty(\RN)} = \|\rho_2(\hat U^* - \psi)\|_{L^\infty(\RN)} \leq C(\epsilon_0)\mathbf{1}_{\{\frac{\epsilon_0}{6} \leq |x| \leq 1\}}.$$
Hence, if $t_0 \in [t_{0,4}(\epsilon, \epsilon_0, \sigma_0), T)$, we obtain 
$$\|U_2(t_\epsilon) - U_2(t_0)\|_{L^\infty(\RN)} \leq \frac{\epsilon}{2}.$$
This follows $\|U(t_\epsilon) - U(t_0)\|_{L^\infty(|x| \geq \frac{\epsilon_0}{6})} \leq \frac{\epsilon}{2}$, which is a contradiction to \eqref{eq:asupD31}. This concludes the proof of Proposition \ref{prop:D3}.
\end{proof}

\subsection{Finite dimensional reduction.}\label{sec:con}
In this subsection, we give the proof of Proposition \ref{prop:redu}, which follows from the \textit{a priori estimates} obtained in Propositions \ref{prop:D1}, \ref{prop:D2} and \ref{prop:D3}. We proceed in two steps: we first show that we can fix $K_0, \delta_0$ and $C_0$ independently from $A$, take $A$ sufficiently large and choose $\epsilon_0, \alpha_0$ and $\eta_0$ in terms of $A$, so that all the bounds given in Definition \ref{def:St} are improved, except for the modes $Q_0$ and $Q_1$. This immediately gives the conclusion of part $(i)$ of Proposition \ref{prop:redu}, which reduces the problem to a finite dimensional one. In second step, we use the dynamics on the modes $Q_0$ and $Q_1$ given in Proposition \ref{prop:D1} to prove part $(ii)$ of Proposition \ref{prop:redu}. 
\subsubsection{Improved controls of $U(t)$ in $\Sc^*(t)$ and conclusion of part $(i)$ of Proposition \ref{prop:redu}.}
We aim at proving that for a suitable choice of the parameters $t_0, K_0, \epsilon_0, \alpha_0$, $A, \delta_0, C_0, C_0'$ and $\eta_0$, the bounds given in Definition \ref{def:St} can be improved. In particular, we want to prove that under the assumption of Proposition \ref{prop:redu}, the following estimates hold:\\
\noindent - \textit{(Improved controls in $\Dc_1$)} For $s_* = -\ln (T - t_*)$,
\begin{align}
i,j \in \{1, \cdots, N\}, \;|Q_{2,i,j}(s_*)| < \frac{A^2 \ln s_*}{s_*^2}, \quad |Q_-(y, s_*)| \leq \frac{A}{2s_*^2}(1 + |y|^3), \label{est:imD1}\\
|Q_e(y,s_*)| \leq \frac{A^2}{2\sqrt{s_*}}, \quad |(\nabla Q)_\bot(y,s_*)| \leq \frac{A}{2s_*^2}(1 + |y|^3).\label{est:impD12}
\end{align}
\noindent - \textit{(Improved controls in $\Dc_2$)} For all $|x| \in \left[\frac{K_0}{4}\sqrt{(T-t_*)|\ln(T-t_*)|}, \epsilon_0 \right]$ and $|\xi| \leq \alpha_0 \sqrt{|\ln \theta(x)|}$, where $\theta(x) = T - t(x)$ and $t(x)$ is defined by \eqref{def:tx},
\begin{align}\label{est:impD2}
|\Uc(x, \xi, \tau_*) - \hat \Uc(\tau_*)| \leq \frac{\delta_0}{2},\quad |\nabla_\xi \Uc(x, \xi, \tau_*)| \leq \frac{C_{0}}{2\sqrt{|\ln \theta(x)|}}, \quad |\nabla^2_\xi \Uc(x,\xi, \tau_*)| \leq \frac{C_{0}'}{2},
\end{align}
where $\tau_* = \frac{t_* - t(x)}{\theta(x)}$.\\
\noindent - \textit{(Improved controls in $\Dc_3$)} For all $|x| \geq \frac{\epsilon_0}{4}$,
\begin{align}\label{est:impD3}
|U(x,t_*) - U(x,t_0)| \leq \frac{\eta_0}{2}, \quad |\nabla U(x,t_*) - \nabla U(x,t_0)| \leq \frac{\eta_0}{2}.
\end{align}
One can see that once these improved estimates are proven, part $(ii)$ of Proposition \ref{prop:redu} immediately follows. Let us start with the estimates in $\Dc_1$. 

- \textit{Proof of \eqref{est:imD1} and \eqref{est:impD12}:} For the estimate on $Q_{2,i,j}$, we argue by contradiction. We assume that  
$$\forall s \in [s_0, s_*), \;\;|Q_{2,i,j}(s)| < \frac{A^2 \ln s}{s^2}, \quad |Q_{2,i,j}(s_*)| = \frac{A^2\ln s_*}{s_*^2}.$$
Let us consider the case $Q_{2,i,j}(s_*) > 0$ (the case $Q_{2,i,j}(s_*) < 0$ is similar), we have
$$Q'_{2,i,j}(s_*) \geq \left.\frac{d}{ds}\left(\frac{A^2\ln s}{s^2}\right)\right|_{s =s_*} = \frac{A^2}{s_*^3} - \frac{2A^2\ln s^*}{s_*^3},$$
on the one hand. On the other hand, we have by \eqref{equ:odev2}, 
$$Q'_{2, i,j}(s_*) \leq \frac{CA}{s_*^3} - \frac{2A^2\ln s_*}{s_*^3},$$
and a contradiction follows if $A \geq C+1$. This proves \eqref{est:imD1} for $Q_{2,i,j}$.

For the improved controls of $Q_-, Q_e$ and $(\nabla Q)_\bot$, we distinguish in two cases:

\noindent - Case 1: $s_* - s_0 \leq \lambda_1$ where $\lambda_1 = \lambda_1(A) > 0$ is fixed later. We apply Proposition \ref{prop:D1} with $\lambda^* = \lambda_1$, $\lambda = s_* - s_0$ and $\sigma = s_0$ to obtain
\begin{align*}
&|Q_-(y,s_*)| \leq C(1 + \lambda_1)s_*^{-2}(1 + |y|^3),\\
&|Q_e(y,s_*)| \leq CK_0^3e^{\lambda_1}(1 + \lambda_1)s_*^{-1/2},\\
&|(\nabla Q)_\bot(y, s_*)| \leq C(1 + \lambda_1)s_*^{-2}(1 + |y|^3).
\end{align*}
To have \eqref{est:imD1} and \eqref{est:impD12}, we need $C(1 + \lambda_1) \leq A/2$ and $CK_0^3(1 + \lambda_1)e^{\lambda_1} \leq A^2/2$, which is possible with $\lambda_1 = \frac{3}{2}\ln A$ for $A$ large enough. 

\noindent - Case 2: $s_* - s_0 \geq \lambda_2$ where $0 < \lambda_2=\lambda_2(K_0, A) \leq \lambda_1$. We apply Proposition \ref{prop:D1} with $\lambda^* = \lambda = \lambda_2$, $\sigma = s_* - \lambda_2$ to obtain
\begin{align*}
&|Q_-(y,s_*)| \leq C(Ae^{-\lambda_2/2} + A^2e^{-\lambda_2^2} + \lambda_2)s_*^{-2}(1 + |y|^3),\\
&|Q_e(y,s_*)| \leq C(A^2e^{-\lambda_2/2} +  AK_0^3\lambda_2e^{\lambda_2})s_*^{-1/2},\\
&|(\nabla Q)_\bot(y, s_*)| \leq C(Ae^{-\lambda_2/2} + C(K_0)C_0e^{-\lambda_2^2} + \sqrt{\lambda_2} + \lambda_2)s_*^{-2}(1 + |y|^3).
\end{align*}
We now fix $\lambda_2$ so that 
$$C(Ae^{-\lambda_2/2} + A^2e^{-\lambda_2^2} + \lambda_2) \leq \frac{A}{2},\quad  C(A^2e^{-\lambda_2/2} +  AK_0^3\lambda_2e^{\lambda_2}) \leq \frac{A^2}{2},$$
and 
$$C(Ae^{-\lambda_2/2} + C(K_0)C_0e^{-\lambda_2^2} + \sqrt{\lambda_2} + \lambda_2) \leq \frac{A}{2},$$
which is possible with $\lambda_2 = \ln(A/(8CK_0^3))$ and $C_0 \leq A^3$, then the conclusion follows for $A$ large enough. This concludes the proof of \eqref{est:imD1} and \eqref{est:impD12}.

- \textit{Proof of \eqref{est:impD2}:} We aim at using Proposition \ref{prop:D2} to prove \eqref{est:impD2}. For this we need to check the hypothesis of Proposition \ref{prop:D3} under the hypothesis of Proposition \ref{prop:redu}. This is done thanks to the following result:
\begin{lemma} \label{lemm:D2} Under the hypothesis of Proposition \ref{prop:redu}, we have for 
$$|x| \in \left[\frac{K_0}{4}\sqrt{(T-t_*)|\ln(T-t_*)|}, \epsilon_0 \right],$$
\begin{itemize}
\item[(i)] For all $|\xi| \leq \frac{7}{4}\alpha_0 \sqrt{|\ln \theta(x)|}$ and  $\tau \in \left[\max\left\{0, \frac{t_0 - t(x)}{\theta(x)}\right\}, \frac{t_* - t(x)}{\theta(x)}\right]$,
\begin{align*}
|\nabla_\xi \Uc(x, \xi, \tau)| \leq \frac{2C_0}{\sqrt{|\ln \theta(x)|}}, \quad |\nabla_\xi^2 \Uc(x, \xi, \tau)| \leq 2C_0', \quad \Uc(x, \xi, \tau) \geq \frac{1}{2}\hat \Uc(\tau).
\end{align*}
\item[(ii)] For all $|\xi| \leq 2\alpha_0 \sqrt{|\ln \theta(x)|}$ and  $\tau = \max\left\{0, \frac{t_0 - t(x)}{\theta(x)}\right\}$,
$$\forall \delta_1 \leq 1, \; |\Uc(x, \xi, \tau) - \hat \Uc(\tau)| \leq \delta_1, \quad |\nabla_\xi \Uc(x, \xi, \tau)| \leq \frac{C_0}{4\sqrt{|\ln \theta(x)|}}.$$
\end{itemize}
\end{lemma}
\begin{proof} Since the proof of Lemma \ref{lemm:D2} is very similar to the one given in \cite{MZnon97} and no new ideas are needed, we refer the interested reader to see Lemma 2.6 at page 1515 in that paper for all details of the proof.
\end{proof}
From Lemma \ref{lemm:D2}, we apply Proposition \ref{prop:D2} with $C_{0,1}^* = \frac{C_0}{4}$, $C_{0,2}^* = 2C_0$, $C_{0,3}^* = 2C_0'$, $\xi_0 = \alpha_0\sqrt{|\ln \theta(\epsilon_0)|}$ with the choice $\alpha_0 \leq 1$ to derive for $|\xi| \leq \alpha_0\sqrt{|\ln \theta(x)|}$, 
$$|\Uc(x, \xi, \tau_*) - \hat \Uc(\tau_*)| \leq \frac{\delta_0}{2}, \quad |\nabla_\xi \Uc(x, \xi, \tau_*)| \leq \frac{C_0}{2\sqrt{|\ln \theta(x)|}}.$$
By a direct parabolic estimate, we see that there exists $t_{0,6}(A) < T$ such that for all $t_0 \in [t_{0,6}, T)$, if 
$$U(t_0) \in \Sc^*(t_0, K_0, \epsilon_0, \alpha_0, A, \delta_0, C_0, C_1, \eta_0, t_0)$$
and $$U(t) \in \Sc^*(t_0, K_0, \epsilon_0, \alpha_0, A, \delta_0, C_0, C_0', \eta_0, t) \quad \forall t \in [t_0, t'],$$
then 
$$U(t') \in \Sc^*(t_0, K_0, \epsilon_0, \alpha_0, A, \delta_0, C_0, 3/4, \eta_0, t').$$
With the choice of $C_0' \geq 2$ and from Proposition \ref{prop:PropInitialdata}, we have $|\nabla_\xi^2\Uc(x, \xi, \tau_*)| \leq \frac{3}{4} \leq \frac{C_0'}{2}$. This concludes the proof of \eqref{est:impD2}, assuming that Lemma \ref{lemm:D2} holds.

- \textit{Proof of \eqref{est:impD3}:} We aim at using Proposition \ref{prop:D3} to improve the estimates in $\Dc_3$. Let us check the hypothesis of Proposition \ref{prop:D3}. From $(ii)$ of Definition \eqref{def:St}, we have for all $t \in [t_0, t_*]$ and $|x| \in [\frac{\epsilon_0}{6}, \frac{\epsilon_0}{4}]$,
\begin{equation}\label{est:Uc012}
|\Uc(x, 0, \tau) - \hat \Uc(\tau)| \leq \delta_0, \quad |\nabla_\xi\Uc(x, 0, \tau)| \leq \frac{C_0}{\sqrt{|\ln \theta(x)|}}, \quad |\nabla^2_\xi \Uc(x, 0, \tau)| \leq C_0',
\end{equation}
where $\tau = \frac{t - t(x)}{\theta(x)}$. 
From the definition \eqref{def:Uc} of $\Uc(x, \xi, \tau)$ and \eqref{est:Uc012}, we obtain $|\nabla^i U(x,t)| \leq \sigma_0(K_0, \epsilon_0, C_0, C_0')$. From \eqref{def:U0d0d1}, we have $U(x,t_0) = \hat U^*(x)$ for $|x| \geq \frac{\epsilon_0}{6}$. Hence, Proposition \ref{prop:D3} applies with $\epsilon = \eta_0/2$ and we obtain for all $t \in [t_0, t_*]$ and $|x| \geq \frac{\epsilon_0}{4}$,
$$|U(x,t) - U(x,t_0)| + |\nabla U(x,t) - \nabla U(x,t_0)| \leq \frac{\eta_0}{2},$$
which concludes the proof of \eqref{est:impD3}. This also completes the proof of part $(i)$ of Proposition \ref{prop:redu}.
\subsubsection{Outgoing transversality of $U(t)$ on $\partial \Sc^*(t)$.}
We give here the proof of part $(ii)$ of Proposition \ref{prop:redu}. The proof simply follows from part $(i)$ of Proposition \ref{prop:redu} and the ODEs \eqref{equ:odev0} and \eqref{equ:odev01}. Indeed, from part $(i)$, we know that for $\omega = \pm 1$, $Q_{0}(s_*) = \frac{\omega A}{s_*^2}$ and $Q_{1,i}(s_*) = \frac{\omega A}{s_*^2}$ for $i \in \{1, \cdots, N\}$. From \eqref{equ:odev0} and \eqref{equ:odev01}, we see that 
\begin{align*}
&\omega Q_0'(s_*) \geq \omega Q_0(s_*) - \frac{C}{s_*^2} \geq \frac{A - C}{s_*^2},\\
\forall i \in \{1, \cdots, N\},\;\; &\omega Q_{1,i}'(s_*) \geq \frac{1}{2}\omega Q_{1,i}(s_*) - \frac{C}{s_*^2} \geq \frac{A/2 - C}{s_*^2}.
\end{align*}
Taking $A$ large enough gives $\omega Q'_{0}(s_*) > 0$ and $\omega Q'_{1,i}(s_*) > 0$, which means that $Q_0$ and $Q_{1,i}$ are traversal outgoing to the bounding curve $s \mapsto \omega As^{-2}$ at $s = s_*$. This concludes the proof of part $(ii)$ and completes the proof of Proposition \ref{prop:redu}.

\appendix
\renewcommand*{\thesection}{\Alph{section}}

\section{Properties of the set $\Vc_{A, K_0}$ and the initial data \eqref{def:U0d0d1}.}\label{sec:preinti}
In this appendix, we give some properties of the shrinking set $\Vc_{A, K_0}$ defined in part $(i)$ of Definition \ref{def:St} as well as the proof of Proposition \ref{prop:PropInitialdata}. Let us start with the following proposition:
\begin{proposition}[Estimates on $Q$ and $\nabla Q$ in $\Vc_{A, K_0}$]\label{prop:propSt} For all $K_0 \geq 1$ and $\epsilon_0 > 0$, there exist $t_{0,2}(K_0, \epsilon_0)$ and $\eta_{0,2}(\epsilon_0) > 0$ such that for all $t_0 \in [t_{0,2}, T)$, $A \geq 1$, $\alpha_0 > 0$, $C_0 > 0$, $\delta_0 \leq \frac{1}{2}|\hat \Uc(1)|$ and $\eta_0 \in (0, \eta_{0,2}]$, we have the following properties: Assume that $U(x,t_0)$ is given by \eqref{def:U0d0d1} and that for some $t \in [t_0,T)$,
$$U(t) \in \Sc^*(t_0, K_0, \epsilon_0, \alpha_0, A, \delta_0, \eta_0, C_0, C_0', t),$$
then, there exists a positive constant $C = C(K_0, C_0)$ such that for all $y \in \RN$ and $s = -\log(T-t)$,\\
$(i)$ (Estimates on $Q$)
$$|Q(y,s)| \leq \frac{CA^2}{\sqrt s}, \quad |Q(y,s)| \leq \frac{CA^2\ln s}{s^2}(|y|^2 + 1) + \frac{CA^2}{s^2}(|y|^3 + 1).$$
$(ii)$ (Estimates on $\nabla Q$)
\begin{align*}
|\nabla Q(y,s)| \leq \frac{CA^2}{\sqrt s}, \quad |\nabla Q(y,s)| \leq \frac{CA^2\ln s}{s^2}(|y|^3 + 1),\quad |(1 - \chi(y,s))\nabla Q(y,s)| \leq \frac{C}{\sqrt s}.
\end{align*}
\end{proposition}
\begin{proof} $(i)$ From part $(i)$ of Definition \ref{def:St}, we have $Q(s) \in \Vc_{A, K_0}(s)$. From the decomposition \eqref{decomqQ}, we have 
\begin{align*}
|Q(y,s)| &\leq \sum_{m = 0}^2|Q_m(s)||h_m(y)| + |Q_-(y,s)| + |Q_e(y,s)|\\
&\leq \frac{CA^2 \ln s}{s^2}(|y^2| + 1) + \frac{A}{s^2}(|y|^3 + 1) + \frac{\|Q_e(s)\|_{L^\infty}(\RN)}{|y|^3 + 1}(|y|^3 + 1)\mathbf{1}_{\{|y| \geq K_0 \sqrt s\}}\\
& \leq \frac{CA^2 \ln s}{s^2}(|y|^2 + 1) + \frac{CA^2}{s^2}(|y|^3 + 1). 
\end{align*}
We also have
$$|Q_b(y,s)| \leq \left\{\sum_{m = 0}^2|Q_m(s)||h_m(y)| + |Q_-(y,s)|\right\}\mathbf{1}_{\{|y| \leq 2K_0 \sqrt s\}} \leq \frac{CA^2 \ln s}{\sqrt s},$$
hence,
\begin{align*}
|Q(y,s)| \leq \|Q_b(s)\|_{L^\infty(\RN)} + \|Q_e(s)\|_{L^\infty(\RN)} \leq \frac{CA^2}{\sqrt{s}},
\end{align*}
which concludes the proof of part $(i)$.\\

$(ii)$ Arguing similarly as for $(i)$, we obtain from part $(i)$ of Definition \ref{def:St} and \eqref{decomQ2}, 
$$|\chi(y,s) \nabla Q(y,s)| \leq \frac{CA^2 \ln s}{s^2}(|y|^3 + 1), \quad |\chi(y,s)\nabla Q(y,s)| \leq \frac{CA^2}{\sqrt{s}}.$$
Note from \eqref{def:Qpsi} that $|\nabla \psi_\alpha(y,s)| \leq \frac{C}{\sqrt s}$ and $\frac{C}{\sqrt s} \leq \frac{C}{s^2}(|y|^3 + 1)$ for $|y| \geq K_0 \sqrt s$ and $K_0 \geq 1$. We have to prove that 
$$|(1 - \chi(y,s)) \nabla (Q + \psi_\alpha)(y,s)| = |(1 - \chi(y,s))\nabla W(y,s)e^{W(y,s)}| \leq \frac{C}{\sqrt s}$$
in order to conclude the proof.
From \eqref{def:changeVar}, this reduces to show that for all $t \geq t_0$ and $|x| \geq r(t) = K_0\sqrt{(T-t)|\log(T-t)|}$, 
\begin{equation}\label{est:nabUeU}
|\nabla_x U(x,t) e^{U(x,t)}| \leq \frac{C}{(T-t)^{\frac{3}{2}}\sqrt {|\ln (T-t)|}}.
\end{equation}
To prove \eqref{est:nabUeU}, we argue as in \cite{MZnon97} by considering two cases: \\
- \textit{Case 1: $|x| \in [r(t), \epsilon_0]$.} In this case, we use the bounds given in part $(ii)$ of Definition \ref{def:St} to prove \eqref{est:nabUeU}. From \eqref{def:Uc}, we have 
$$U(x,t) = -\ln \theta(x) + \Uc(x, 0, \tau(x,t))$$
and 
$$\nabla_x U(x,t) = \theta(x)^{-\frac 12}\nabla_\xi \Uc(x, 0, \tau(x,t)),$$
where $\theta(x) = T -t(x)$, $\tau(x,t) = \frac{t - t(x)}{\theta(x)}$ and $t(x)$ is uniquely determined by \eqref{def:tx}. Therefore, 
$$|\nabla_x U(x,t) e^{U(x,t)}|  = \theta(x)^{-\frac{3}{2}}\nabla_\xi \Uc(x,0, \tau(x,t)) e^{\Uc(x,0, \tau(x,t))}.$$
Using part $(ii)$ of Definition \ref{def:St}, we have for $|x| \in [r(t),\epsilon_0]$,\\
$$|\Uc(x,0,\tau(x,t)) - \hat \Uc(\tau(x,t))| \leq \delta_0, \quad |\nabla_\xi \Uc(x,0,\tau(x,t))| \leq \frac{C_0}{\sqrt{|\log(\theta(x))|}}.$$
Since $\delta_0 \leq \frac{1}{2}|\hat \Uc(1)| \leq \frac{1}{2}|\hat{\Uc}(\tau)|$, we have from \eqref{def:tx},
$$|\nabla_x U(x,t) e^{U(x,t)}| \leq \frac{C(K_0, C_0)}{\theta(x)^\frac 32 \sqrt{|\log \theta(x)|}} \leq \frac{C(K_0, C_0)}{\theta(r(t))^\frac 32 \sqrt{|\ln \theta(r(t))|}}.$$
Since $r(t) \to 0$ as $t \to T$, we have from \eqref{def:tx}, 
$$\theta(r(t)) \sim \frac{2}{K_0^2}\frac{r(t)^2}{|\ln r(t)|} \quad \text{and}\quad \ln \theta(r(t)) \sim \ln r(t) \quad \text{as} \;\; t \to T.$$
Recall that $r(t) = K_0\sqrt{(T-t)|\ln (T-t)|}$ which gives
$$\frac{C(K_0, C_0)}{\theta(r(t))^\frac 32 \sqrt{|\ln \theta(r(t))|}} \sim \frac{C'(K_0, C_0)}{(T-t)^\frac 32 |\ln (T-t)|}.$$
This concludes the proof of \eqref{est:nabUeU} for $|x| \in [r(t), \epsilon_0]$.\\

\noindent - \textit{Case 2: $|x| \geq \epsilon_0$.} We use here the information contained in $(iii)$ of Definition \ref{def:St}, which asserts that 
$$|\nabla_x^i (U(x,t) - U(x,t_0))| \leq \eta_0, \quad i = 0, 1, \quad \forall |x| \geq \epsilon_0.$$
Let 
$$\eta_{0,2}(\epsilon_0) = \frac{1}{2}\min \left\{\min_{|x| \geq \epsilon_0}|U(x,t_0)|, \min_{|x| \geq \epsilon_0}|\nabla U(x,t_0)|\right\}.$$
From \eqref{def:U0d0d1}, we have $\eta_{0,2}(\epsilon_0) > 0$. Hence, for $\eta_0 \in (0, \eta_{0,2}]$, we obtain for $|x| \geq \epsilon_0$,
$$\left|\nabla_x U(x,t) e^{U(x,t)}\right| \leq C \left|\nabla_x U(x,t_0) e^{U(x,t_0)}\right| \leq C \left|\nabla_x U^*(x) e^{U^*(x)}\right|,$$
where $\hat U^*$ is defined by \eqref{def:Uhatstar}. Note that $\left|\nabla_x U^*(x) e^{U^*(x)}\right| \leq C(\epsilon_0)$ for $|x| \geq \epsilon_0$.  Therefore, if $t_0 \in [t_{0,2},T)$ for some $t_{0,2} < T$, we derive \eqref{est:nabUeU} for $t = t_0$ and $|x| \geq \epsilon_0$, hence, for $t \geq t_0$ and $|x| \geq \epsilon_0$. This completes the proof of Proposition \ref{prop:propSt}.
\end{proof}

We now turn to the proof of Proposition \ref{prop:PropInitialdata} which is a direct consequence of the following lemma:
\begin{lemma}\label{lemm:propinit} There exists $K_{0,1} > 0$ such that for each $K_0 \geq K_{0,1}$ and $\delta_{0,1} > 0$, there exist $\alpha_{0,1}(K_0, \delta_{0,1}) > 0$, $C_{0,1}(K_0) > 0$ and $C_{0,1}'(K_0) > 0$ such that for all $\alpha_0 \in (0, \alpha_{0,1}]$, there exists $\epsilon_{0,1}(K_0, \delta_{0,1}, \alpha_0) > 0$ such that for all $ \epsilon_0 \in  (0,\epsilon_{0,1}]$ and $A \geq 1$, there exists $t_{0,1}(K_0, \delta_{0,1}, \epsilon_0, A, C_{0,1}) < T$ such that for all $t_0 \in [t_{0,1}, T)$, there exists a subset $\mathcal{D}_{t_0, A} \subset \Rb \times \RN$ with the following properties. If $U(x,t_0)$ is defined by \eqref{def:U0d0d1}, then:\\

\noindent $(I)$ For all $(d_0,d_1) \in \Dc_{t_0,A}$, we have $U_{d_0,d_1}(x,t_0)$ defined by \eqref{def:U0d0d1} belongs in\\
$\Sc^*(t_0, K_0, \epsilon_0, \alpha_0, A, \delta_{0,1}, \eta_0 = 0, C_{0,1}, C_{0,1}', t_0)$. More precisely,

$(i)$ \textit{Estimates in $\Dc_1$:} $\bar Q(s_0) \in \Vc_{A, K_0}(s_0)$ with strict inequalities, except for $(\bar Q_0, \bar Q_1)(s_0)$ in the sense that
$$
\left|\bar Q_0(s) -\frac{Ad_0}{s_0^2} \right| \leq Ce^{-s_0},\quad \left|\bar Q_{1,i}(s_0) - \frac{Ad_{1,i}}{s_0^2}\right|\leq Ce^{-s_0},\quad i \in \{1, \cdots, N\},$$
$$|\bar Q_{2,ij}(s_0)| \leq \frac{\ln(s_0)}{s_0^2},  \quad  \forall i,j \in \{1, \cdots, N\},$$
$$|\bar Q_-(y,s_0)| \leq \frac{1}{s_0^2}(|y|^3 + 1),\quad |(\nabla_y \bar Q)_\bot (y,s_0)| \leq \frac{1}{s_0^2}(|y|^3 + 1), \quad \forall y\in \RN,$$
$$\|\bar Q_e(s_0)\|_{L^\infty(\RN)} = 0,$$
where $\bar Q_m$ ($m = 0, 1, 2$), $\bar Q_-$, $\bar Q_\bot$ and $\bar Q_e$ are defined as in \eqref{decomqQ} and \eqref{decomQ2}.

$(ii)$ \textit{Estimates in $\Dc_2$:} For all $|x| \in \left[\frac{K_0}{4}\sqrt{|\ln(T-t_0)|(T-t_0)}, \epsilon_0\right]$, $\tau_0 = \tau_0(x,t_0) = \frac{t_0 - t(x)}{\theta(x)}$ and $|\xi| \leq \alpha_0 \sqrt{\ln \theta(x)}$, 
$$\left|\Uc(x,\xi, \tau_0) - \hat \Uc(\tau_0)\right| \leq \delta_{0,1}, \;\; |\nabla_\xi \Uc(x, \xi, \tau_0)| \leq \frac{C_{0,1}}{\sqrt{|\ln \theta(x)|}}, \;\; |\nabla^2_\xi \Uc(x, \xi, \tau_0)| \leq C_{0,1},$$
where $\Uc$, $\hat \Uc$, $t(x)$ and $\theta(x)$ are defined in \eqref{def:Uc}, \eqref{def:solUc}, \eqref{def:tx} and \eqref{def:thetax} respectively.\\

\noindent $(II)$ 
\begin{align*}
(d_0,d_1) \in \Dc_{t_0,A} &\Longleftrightarrow (\bar Q_0(s_0), \bar Q_1(s_0)) \in \hat \Vc_A(s_0),\\
(d_0,d_1) \in \partial \Dc_{t_0,A} &\Longleftrightarrow (\bar Q_0(s_0), \bar Q_1(s_0)) \in \partial \hat \Vc_A(s_0).
\end{align*}
\end{lemma}
\begin{proof} Part $(II)$ directly follows from item $(i)$ of part $(I)$. Since we have almost the same definition of $\Vc_{A,K_0}$ given part $(i)$ of Definition \ref{def:St} and the definition of initial data \eqref{def:barQs0} as those defined in \cite{MZnon97}, the proof of item $(i)$ is completely analogous as those in \cite{MZnon97}. For this reason, we kindly refer the interested reader to see in particular pages 1535-1536 in that paper for all details of these computations. Because our definition \eqref{def:U0d0d1} is different from those defined in \cite{MZnon97} in the region $\Dc_2$, we only deal with the proof of item $(ii)$.

 Let us consider $t_0 < T$, $K_0$, $\epsilon_0$, $\alpha_0$, $\delta_{0,1}$, $C_{0,1}$ and $C_{0,1}'$, and prove that if these constants are suitably chosen, then for $|x| \in \left[r_0, \epsilon_0\right]$ and $|\xi| \leq 2\alpha_0 \sqrt{|\ln \theta(x)|}$, where $r_0 = \frac{K_0}{4}\sqrt{\theta_0 |\ln \theta_0|}$, $\theta_0 = T - t_0$, $\theta(x) = T - t(x)$ and $t(x)$ is defined as in \eqref{def:tx}, we have 
\begin{equation*}
\left|\Uc(x,\xi, \tau_0) - \hat \Uc(\tau_0)\right| \leq \delta_{0,1}, \;\; |\nabla_\xi \Uc(x, \xi, \tau_0)|\leq \frac{C_{0,1}(K_0)}{\sqrt{|\log\theta(x)|}}, \;\; |\nabla_\xi^2 \Uc(x,\xi, \tau_0)| \leq C_{0,1}'(K_0),
\end{equation*}
where $\tau_0 = \frac{t_0 - t(x)}{T - t(x)}$ and $\hat \Uc$ is given by \eqref{def:solUc}.

We remark from \eqref{def:tx} that if $\alpha_0 \leq \frac{K_0}{16}$ and $\epsilon_0 \leq \frac 23 C(a_1, \alpha)$, where $C(a_1, \alpha)$ is introduced in \eqref{def:Uhatstar}, then for $|x| \in [r_0, \epsilon_0]$ and $|\xi| \leq 2\alpha_0\sqrt{|\ln \theta(x)|}$, we have $|\xi \sqrt{\theta(x)}| \leq \frac{|x|}{2}$, hence 
$$ \frac{r_0}{2} \leq \frac{|x|}{2} \leq |x + \xi\sqrt{\theta(x)}| \leq \frac{3|x|}{2} \leq C(a_1, \alpha).$$
This implies that 
$$\chi(16(x + \xi\sqrt{\theta(x)})\sqrt{\theta_0}, -\ln \theta_0) \chi_1(x + \xi\sqrt{\theta(x)}, t_0) = 0,$$
where $\chi$ and $\chi_1$ are defined by \eqref{def:chi} and \eqref{def:chi1} respectively. Hence, from \eqref{def:U0d0d1}, \eqref{def:Uc}, \eqref{def:Uhatstar} and \eqref{def:Phi_alpha}, we write
\begin{align}
\Uc(x, \xi, \tau_0) &= \ln\left(\frac{(8 + 8\alpha)\theta(x)|\ln |x + \xi\sqrt{\theta(x)}||}{|x + \xi\sqrt{\theta(x)}|^2} \right)(1 - \chi_1(x + \xi\sqrt{\theta(x)}), t_0)\nonumber\\
&-\ln\left[\frac{\theta_0}{\theta(x)}\left(1 + \frac{|x + \xi\sqrt{\theta(x)}|^2}{(4 + 4\alpha)\theta_0|\ln \theta_0|}\right)\right]\chi_1(x + \xi\sqrt{\theta(x)}), t_0)\nonumber\\
& -\frac{N}{(2 + 2\alpha)\ln \theta_0}\chi_1(x + \xi\sqrt{\theta(x)}), t_0)\label{def:Ucxitau0}\\
& := (I)(1 - \chi_1(x + \xi\sqrt{\theta(x)}), t_0)) + (II + III)\chi_1(x + \xi\sqrt{\theta(x)}), t_0).\nonumber
\end{align}

From \eqref{def:tx}, we remark that 
\begin{equation}\label{eq:relthetaxx}
\ln \theta(x) \sim 2\ln|x|, \quad \theta(x) \sim \frac{8}{K_0^2}\frac{|x|^2}{|\ln |x||} \quad \text{as $|x| \to 0$},
\end{equation}
from which we have for a fixed number $K_0 > 0$,
\begin{equation}\label{eq:rer0R0theta}
\theta(r_0) \sim \theta_0, \; \theta(R_0) \sim \frac{16}{K_0^2}\theta_0|\ln \theta_0|, \; \theta(2R_0) \sim \frac{64}{K_0^2}\theta_0|\ln \theta_0|, \; R_0 = \sqrt{\theta_0}|\ln \theta_0|.
\end{equation}

\noindent \textit{- Estimate on $\Uc$}. By linearity and the definition \eqref{def:chi1} of $\chi_1$, it is enough to prove that for $|x| \in [R_0, \epsilon_0]$ and $|\xi| \leq 2\alpha_0 \sqrt{|\ln \theta(x)|}$,
\begin{align}\label{est:IUc}
\left|(I) - \hat \Uc(\tau_0) \right| \leq \frac{\delta_{0,1}}{2},
\end{align}
and for $|x| \in [r_0, 2R_0]$ and $|\xi| \leq 2\alpha_0 \sqrt{|\ln\theta(x)|}$,
\begin{align}\label{est:IIUc}
\left|(II) - \hat \Uc(\tau_0) \right| \leq \frac{\delta_{0,1}}{2}.
\end{align}
We start with \eqref{est:IUc}. From \eqref{def:tx} and \eqref{def:solUc}, we write
\begin{align*}
|(I) - \hat \Uc(\tau_0)| &= \left|-\ln\left(\frac{(8 + 8\alpha)|\ln |x + \xi\sqrt{\theta(x)}||}{|\frac{K_0}{4}\sqrt{|\ln \theta(x)|} + \xi|^2} \right)^{-1} + \ln\left(\frac{\theta_0}{\theta(x)} + \frac{K_0^2/16}{4 + 4\alpha}\right) \right|\\
&\leq C(K_0)\left|\frac{\theta_0}{\theta(x)} + \frac{K_0^2}{16(4 + 4\alpha)} - \frac{|\frac{K_0}{4}\sqrt{|\ln \theta(x)|} + \xi|^2}{(4 + 4\alpha)|\ln |x + \xi\sqrt{\theta(x)}|^2|} \right|\\
& \leq C(K_0)\left[\left|\frac{\theta_0}{\theta(R_0)}\right| +\left| \frac{|\frac{K_0}{4}\sqrt{|\ln \theta(x)|} + 2\alpha_0 \sqrt{|\ln \theta(x)|}|^2}{|\ln |x +2\alpha_0\sqrt{\theta(x)|\ln \theta(x)|}|^2|} - \frac{K_0^2}{16}\right|\right]\\
& \leq C(K_0)\left[\frac{1}{|\ln \theta_0|} +\left|\frac{|\ln \theta(x)|}{|\ln|x|^2 + \ln(1 + 8\alpha_0/K_0)^2|}\left(\frac{K_0}{4} + 2\alpha_0\right)^2 - \frac{K_0^2}{16}\right|\right].
\end{align*}
Since $|x| \leq \epsilon_0$ and $\ln |x|^2 \sim \ln \theta(x)$ as $|x| \to 0$, we find $\alpha_{0,1}(K_0, \delta_{0,1})$ and $t_{0,1}(K_0, \delta_{0,1})$ such that for each $\alpha_0 \leq \alpha_{0,1}$, there is $\epsilon_{0,1}(K_0, \delta_{0,1}, \alpha_0)$ such that for all $\epsilon_0 \leq \epsilon_{0,1}$ and $t_0 \geq t_{0,1}$, the estimate \eqref{est:IUc} holds. 

We now deal with \eqref{est:IIUc}. From \eqref{def:solUc} and \eqref{def:tx}, we write 
\begin{align*}
|(II) - \hat \Uc(\tau_0)| &= \left|\ln\left(\frac{\theta_0}{\theta(x)} + \frac{K_0^2/16}{4 + 4\alpha}\right) - \ln\left(\frac{\theta_0}{\theta(x)} + \frac{|x + \xi\sqrt{\theta(x)} |^2}{(4 + 4\alpha)\theta(x)|\ln \theta_0|} \right) \right|\\
&\quad \leq C\left|\frac{|x + \xi\sqrt{\theta(x)}|^2}{\theta(x)|\ln \theta_0|} - \frac{K_0^2}{16} \right|\leq CK_0^2\left|\left|\sqrt{\frac{\ln \theta(x)}{\ln \theta_0}} + \frac{4\xi}{K_0\sqrt{|\ln \theta_0|}} \right|^2  - 1\right|.
\end{align*}
Since $|x| \in [r_0, 2R_0]$ and $|\xi| \leq 2\alpha_0\sqrt{|\ln \theta(x)|}$, we have 
\begin{align*}
\frac{\ln \theta(2R_0)}{\ln \theta_0}\left(1 - \frac{8\alpha_0}{K_0} \right)^2 - 1 &\leq \left(\sqrt{\frac{\ln \theta(x)}{\ln \theta_0}} + \frac{4\xi}{K_0\sqrt{|\ln \theta_0|}} \right)^2 -1\\
&\qquad \qquad  \leq \frac{\ln \theta(r_0)}{\ln \theta_0}\left(1 + \frac{8\alpha_0}{K_0} \right)^2 - 1.
\end{align*}
Together with \eqref{eq:rer0R0theta}, we find $\alpha_{0,2}(K_0, \delta_{0,1})$ and $t_{0,2}(K_0, \delta_{0,1}) < T$ such that for all $\alpha_0 \leq \alpha_{0,2}$ and $t_0 \geq t_{0,2}$, the estimate \eqref{est:IIUc} holds.\\

\noindent - \textit{Estimate on $\nabla_\xi \Uc$.} From \eqref{def:Ucxitau0}, we write
\begin{align*}
\nabla_\xi \Uc(x, \xi, \tau_0) &= \nabla_\xi(I)(1 - \chi_1(x + \xi\sqrt{\theta(x)}), t_0))\\
&\quad  + \nabla_\xi(II)\chi_1(x + \xi\sqrt{\theta(x)}), t_0)\\
&\qquad + (II - I)\sqrt{\theta(x)}\nabla_x \chi_1(x + \xi\sqrt{\theta(x)}), t_0).
\end{align*}
By the definition \eqref{def:chi1} of $\chi_1$, it is enough to prove the followings in order to obtain the estimate on $\nabla_\xi \Uc$:\\
- for $|x| \in [R_0, \epsilon_0]$ and $|\xi| \leq 2\alpha_0 \sqrt{|\ln \theta(x)|}$,
\begin{equation}\label{est:nabI}
|\nabla_\xi (I)| \leq \frac{C(K_0)}{\sqrt{|\ln \theta(x)|}},
\end{equation}
- for $|x| \in [r_0, 2R_0]$ and $|\xi| \leq 2\alpha_0 \sqrt{|\ln \theta(x)|}$,
\begin{equation}\label{est:nabII}
|\nabla_\xi (II)| \leq \frac{C(K_0)}{\sqrt{|\ln \theta(x)|}},
\end{equation}
- for $|x| \in [R_0, 2R_0]$ and $|\xi| \leq 2\alpha_0 \sqrt{|\ln \theta(x)|}$,
\begin{equation}\label{est:nabIII}
|I - II|\sqrt{\theta(x)}|\nabla_x \chi_1(x + \xi\sqrt{\theta(x)}), t_0)| \leq \frac{C(K_0)}{\sqrt{|\ln \theta(x)|}}.
\end{equation}
For \eqref{est:nabI} and \eqref{est:nabII}, a straightforward computation yields
\begin{align*}
|\nabla_\xi (I)| &= \left|\nabla_\xi \ln\left(\frac{(8 + 8\alpha)|\ln |x + \xi\sqrt{\theta(x)}||}{|\frac{K_0}{4}\sqrt{|\ln \theta(x)|} + \xi|^2} \right) \right|\\
&\quad  \leq \frac{C\sqrt{\theta(x)}}{|x + \xi\sqrt{\theta(x)}||\ln|x + \xi \sqrt{\theta(x)}||} + \frac{C}{|\frac{K_0}{4}\sqrt{|\ln \theta(x)|} + \xi|} \leq \frac{C}{\sqrt{|\ln \theta(x)|}},
\end{align*}
and 
\begin{align*}
|\nabla_\xi (II)| &= \left|\nabla_\xi \ln\left(1 + \frac{|x + \xi\sqrt{\theta(x)}|^2}{(4 + 4\alpha)\theta_0|\ln \theta_0|}\right) \right|\\
&\quad \leq \frac{C\sqrt{\theta(x)}}{|x + \xi\sqrt{\theta(x)}|} \leq \frac{C}{|\frac{K_0}{4} \sqrt{|\ln \theta(x)|} + \xi |} \leq \frac{C}{\sqrt{|\ln \theta(x)|}}.
\end{align*}
 As for \eqref{est:nabIII}, we note from \eqref{def:chi1} that 
 $$|\nabla_x \chi_1(x + \xi\sqrt{\theta(x)}), t_0)| \leq \frac{C}{\sqrt{\theta_0}|\ln \theta_0|}.$$
 From \eqref{eq:rer0R0theta}, we have  $\theta(x) \sim C(K_0)\theta_0|\ln \theta_0|$ and $|\ln \theta(x)| \sim |\ln \theta_0|$ for $x \in [R_0, 2R_0]$. Therefore, the estimate \eqref{est:nabIII} follows once the following is proved: for $|x| \in [R_0, 2R_0]$ and $|\xi| \leq 2\alpha_0\sqrt{|\ln \theta(x)|}$,
 $$|(I)| + |(II)| \leq C.$$
To this end, we note from \eqref{eq:rer0R0theta} and \eqref{eq:relthetaxx} that $|\ln |x + \xi\sqrt{\theta(x)}|| = |\ln |x| + \ln (1 + \frac{\xi\sqrt{\theta(x)}}{|x|})| \sim \frac{1}{2}|\ln \theta(x)|$, and that $|x + \xi\sqrt{\theta(x)}|^2 \sim C(K_0)\theta(x)|\ln \theta(x)|$. This follows
$$|(I)| = \left|\ln \left(\frac{(8 + 8\alpha)\theta(x)|\ln |x + \xi \sqrt{\theta(x)}||}{|x + \xi\sqrt{\theta(x)}|^2}\right)\right|\sim \left|\ln \left(\frac{(4 + 4\alpha)\theta(x)|\ln \theta(x)|}{C(K_0)\theta(x)|\ln \theta(x)|}\right)\right| \leq C,$$
and 
$$|(II)| \sim \left|\ln \left[\frac{1}{|\ln \theta_0|}\left(\frac{C(K_0)\theta(x)|\ln \theta(x)|}{(4 + 4\alpha)\theta_0|\ln \theta_0|}\right) \right] \right| \sim \left|\ln C(K_0)\right| \leq C,$$
which concludes the proof of \eqref{est:nabIII}. The expected estimate on $\nabla_\xi \Uc$ then follows from \eqref{est:nabI}, \eqref{est:nabII} and \eqref{est:nabIII}.

In the same way, we can show that it $t_0 \geq t_{0,3}(K_0, \epsilon_0, C_{0,1})$, then for $|x| \in [r_0, \epsilon_0]$ and $|\xi| \leq 2\alpha_0\sqrt{|\ln \theta(x)|}$, we have $|\nabla^2_\xi \Uc(x, \xi, \tau_0)| \leq C_{0,1}$. This completes  the proof of Lemma \ref{lemm:propinit} as well as the proof of Proposition \ref{prop:PropInitialdata}.
\end{proof}

\section{No blowup under some threshold for equation \eqref{equ:problem}.} \label{sec:noblowup}
We give in this appendix the proof of Proposition \ref{prop:noblowup} whose proof uses ideas given in \cite{GKcpam89} treated for equation \eqref{equ:sheGra} without the gradient term $(\alpha = 0)$. The proof is based on the following integral equations for localizations of $e^u$ and $\nabla u$:
\begin{lemma} \label{lemm:vgtau} Let $r > 0$ and $\phi_r$ be a smooth function supported on $B_r = \{x \in \RN, |x| < r\}$ such that $\phi_r = 1$ on $B_{r/2}$ and $0 \leq \phi_r \leq 1$. Let $w_r = \phi_r e^u$ and $g_r = \phi_r \nabla u$, where $u$ satisfies equation \eqref{ine:pb}. Then, we have the following: 
\begin{align}
\|w_r(\tau)\|_{L^\infty} &\leq \|w(0)\|_{L^\infty} + C \int_{0}^\tau\left(1 + \frac{1}{\sqrt{\tau -s}}\right)\left\|e^{u(s)}\right\|_{L^\infty(B_r)}ds\nonumber\\
& + C\int_0^\tau \left(\left\|e^{u(s)}\right\|_{L^\infty(B_r)} + \|\nabla u(s)\|_{L^\infty(B_r)}^2\right)\|w_r(s)\|_{L^\infty} ds,\label{equ:wtau}\\
\|g_r(\tau)\|_{L^\infty} &\leq \|g(0)\|_{L^\infty} + C \int_{0}^\tau\left(1+ \frac{1}{\sqrt{\tau -s}}\right)\|\nabla u(s)\|_{L^\infty(B_r)}ds + C\int_0^\tau \|\nabla u(s)\|^2_{L^\infty(B_r)}ds\nonumber\\
& + C\int_0^\tau \left(\left\|e^{u(s)}\right\|_{L^\infty(B_r)} + \frac{1}{\sqrt{\tau - s}}\|\nabla u(s)\|_{L^\infty(B_r)}\right) \|g_r(s)\|_{L^\infty}ds.\label{equ:gtau}
\end{align}
\end{lemma}
\begin{proof} We only deal with \eqref{equ:wtau} since \eqref{equ:gtau} follows similarly. By the definition, we see that $w_r$ satisfies the following equation:
$$\partial_\tau w_r - \Delta w_r = \big(e^u \Delta \phi_r - 2\nabla \cdot (e^u \nabla \phi_r)\big) + e^u\phi_r (\partial_\tau u - \Delta u - |\nabla u|^2),$$
hence, the semigroup representation formula for $w_r$ gives
\begin{align}
w_r(\tau) = e^{\tau \Delta}w_r(0) &+ \int_0^\tau e^{(\tau - s)\Delta}\big(e^u \Delta \phi_r - 2\nabla \cdot (e^u \nabla \phi_r)\big) ds \nonumber\\
&+ \int_0^\tau e^{(\tau - s)\Delta}e^u\phi_r (\partial_\tau u - \Delta u - |\nabla u|^2)ds,\label{equ:wtauint}
\end{align}
for $0 < \tau < 1$, where $e^{\theta \Delta}$ is the semigroup associated with the heat equation in $\RN$ with the following well known properties:
\begin{equation}\label{equ:regularityeffect}
\|e^{\theta \Delta }f\|_{L^\infty} \leq \|f\|_{L^\infty}, \quad \|\nabla e^{\theta \Delta}f\|_{L^\infty} \leq \frac{1}{\sqrt{\theta}}\|f\|_{L^\infty}, \quad \forall \theta > 0, \; f \in W^{1, \infty}(\RN).
\end{equation}
Using this regularity effect and the definition of $\phi_r$, the middle term in the right hand side of \eqref{equ:wtauint} is bounded by
\begin{align*}
\int_0^\tau \left\|e^u \Delta \phi_r - 2\nabla \cdot (e^u \nabla \phi_r)\right\|_{L^\infty}ds &\leq C\int_{0}^\tau \left(1 + \frac{1}{\sqrt{\tau - s}}\right)\left\|e^{u(s)}\right\|_{L^\infty(B_r)}ds.
\end{align*}
From inequality \eqref{ine:pb}, the last term is dominated by
\begin{align*}
C\int_0^\tau& \| e^{(\tau -s)\Delta}e^u\phi_r(1 + e^u + |\nabla u|^2)\|_{L^\infty} ds \\
&\quad \leq C\int_0^\tau \left\|e^{u(s)}\right\|_{L^\infty(B_r)}ds\\
&\qquad  + C\int_0^\tau \left(\left\|e^{u(s)}\right\|_{L^\infty(B_r)} + \|\nabla u(s)\|_{L^\infty(B_r)}^2 \right)\|w_r(s)\|_{L^\infty} ds.
\end{align*}
This concludes the proof of \eqref{equ:wtau} as well as Lemma \ref{lemm:vgtau}.
\end{proof}

Before going to the proof of Proposition \ref{prop:noblowup}, it is convenient to recall the two following lemmas from \cite{GKcpam89} which will be used in the proof. The first lemma gives estimates on an integration.
\begin{lemma} \label{lemm:GK1} For $0 < \alpha < 1$, $\theta >0$ and $\theta < h < 1$, the integral 
$$I(h) = \int_h^1(s - h)^{-\alpha}s^{-\theta}ds$$
satisfies
\begin{itemize}
\item[(i)] $I(h) \leq \left(\frac{1}{1 - \alpha} + \frac{1}{\alpha + \theta - 1}\right)h^{1 - \alpha - \theta}\;$ if $\;\alpha + \theta > 1$.
\item[(ii)] $I(h) \leq \frac{1}{1 - \alpha} + |\ln h|\;$ if $\;\alpha + \theta = 1$. 
\item[(iii)] $I(h) \leq \frac{1}{1 - \alpha - \theta}\;$ if $\;\alpha + \theta < 1$.
\end{itemize}
\end{lemma}
\begin{proof} See Lemma 2.2, page 851 in \cite{GKcpam89}.
\end{proof}
The second lemma is a version of Gronwall's inequality.
\begin{lemma} \label{lemm:Gronwall} If $y(t), r(t)$ and $q(t)$ are continuous functions defined on $[t_0,t_1]$ such that
$$y(t)\leq y_0 +\int_{t_0}^t q(s)ds + \int_{t_0}^tr(s)y(s)ds, \quad t_0 \leq t \leq t_1,$$
then 
$$y(t) \leq e^{\int_{t_0}^t r(s)ds} \left[y_0 + \int_{t_0}^t q(s) e^{-\int_{t_0}^s r(\tau)d\tau}ds\right].$$
\end{lemma}
\begin{proof} See Lemma 2.3, page 852 in \cite{GKcpam89}.
\end{proof}
We now  give the proof of Proposition \ref{prop:noblowup}. 
\begin{proof}[Proof of Proposition \ref{prop:noblowup}] We proceed in two steps:\\
- \textit{Step 1}:  We first apply Lemma \ref{lemm:vgtau} with $r = 1$ and use the assumption \eqref{con:ineUnU} to get 
$$
\|w_1(\tau)\|_{L^\infty} \leq \epsilon + C\epsilon\int_0^\tau \frac{ds}{(1 - s)\sqrt{\tau - s}} + C\epsilon \int_0^\tau \frac{1}{1 - s}\|w_1(s)\|_{L^\infty}ds,
$$
and 
\begin{align*}
\|g_1(\tau)\|_{L^\infty} &\leq \epsilon + C\epsilon \int_{0}^\tau \left(\frac{1}{\sqrt{(\tau - s)(1 - s)}} + \frac{1}{1 - s}\right)ds \\
&+ C\epsilon\int_0^\tau \left(\frac{1}{\sqrt{(\tau - s)(1 - s)}} + \frac{1}{1 - s}\right)\|g_1(s)\|_{L^\infty}ds.
\end{align*}
Now, applying Lemma \ref{lemm:Gronwall} to $\|w_1(\tau)\|_{L^\infty}$ yields
$$\|w_1(\tau)\|_{L^\infty} \leq (1 - \tau)^{-C\epsilon} \left[\epsilon + C\epsilon \int_0^\tau \frac{ds}{(1 - s)^{1 - C\epsilon}\sqrt{\tau - s}}\right].$$
Using item $(i)$ of Lemma \ref{lemm:GK1}, we have
$$\int_0^\tau \frac{ds}{(1 - s)^{1 - C\epsilon}\sqrt{\tau - s}} = \int_{1 - \tau}^1 \frac{ds}{s^{1 - C\epsilon}\sqrt{s - (1 - \tau)}} \leq C(1 - \tau)^{-\frac{1}{2} + C\epsilon}.$$
Hence, if $\epsilon$ is small enough such that $C\epsilon < \frac{1}{2}$, we have
$$\|w_1(\tau)\|_{L^\infty} \leq C\epsilon(1 - \tau)^{- \frac{1}{2}}, \quad \forall \tau \in[0, 1).$$
Similarly, applying Lemma \ref{lemm:Gronwall} to $\|g_1(\tau)\|_{L^\infty}$ and using item $(ii)$ of Lemma \ref{lemm:GK1} to estimate the integral $\int_0^\tau \frac{ds}{\sqrt{(\tau - s)(1 - s)}} \leq C + |\ln (1 - \tau)|$, we then obtain 
$$\|g_1(\tau)\|_{L^\infty} \leq C(1 - \tau)^{-C\epsilon}\left[\epsilon + C\epsilon\int_0^\tau \left(\frac{1}{(1 - s)^{1/2 - C\epsilon}\sqrt{\tau - s}} + \frac{1}{(1 - s)^{1 - C\epsilon}} \right)ds\right].$$
Using item $(iii)$ of Lemma \ref{lemm:GK1}, we then get 
$$\|g_1(\tau)\|_{L^\infty} \leq C\epsilon(1 - \tau)^{-C\epsilon}, \quad \forall \tau \in [0,1).$$
Therefore, we have by the definition of $\phi_1$ that 
\begin{equation}\label{equ:newEstw1g1}
e^{u(\xi,\tau)} \leq C\epsilon (1 - \tau)^{-\frac{1}{2}}\;\;\text{and}\;\;|\nabla u(\xi,\tau)| \leq C\epsilon(1 - \tau)^{-C\epsilon}, \quad \forall |\xi| < \frac{1}{2}, \; \forall \tau \in [0, 1). 
\end{equation}

\noindent - \textit{Step 2}: Applying Lemma \ref{lemm:vgtau} with $r = \frac{1}{2}$, and using \eqref{equ:newEstw1g1}, we then get 
$$\|w_{1/2}(\tau)\|_{L^\infty} \leq C\epsilon + C\epsilon\int_0^\tau \frac{ds}{\sqrt{(\tau - s)(1 - s)}} + C \int_0^\tau \frac{1}{\sqrt{1 - s}} + \frac{1}{(1 - s)^{2C\epsilon}}\|w_{1/2}(s)\|_{L^\infty}ds,$$
and 
\begin{align*}
\|g_{1/2}(\tau)\|_{L^\infty} \leq C\epsilon &+ C\epsilon\int_0^\tau \left(\frac{1}{(1 - s)^{C\epsilon}\sqrt{\tau - s}} + \frac{1}{(1 - s)^{2C\epsilon}}\right)ds \\
& + C\epsilon\int_0^\tau\left( \frac{1}{\sqrt{1 - s}} + \frac{1}{(1 - s)^{C\epsilon}\sqrt{\tau - s}}\right)\|g_{1/2}(s)\|_{L^\infty}ds.
\end{align*}
From Lemmas \ref{lemm:Gronwall} and \ref{lemm:GK1}, we deduce  
$$\|w_{1/2}(\tau)\|_{L^\infty} \leq C\epsilon(1 + |\ln(1 - \tau)|)\; \text{and} \;\; \|g_{1/2}(\tau)\|_{L^\infty} \leq C\epsilon, \quad \forall \tau \in [0, 1).$$
Since $|\ln (1 -\tau)| \leq (1 - \tau)^{-\epsilon}$ for any $\epsilon > 0$, we then have by the definition of $\phi_{1/2}$
\begin{equation}\label{est:Newu12g12}
e^{u(\xi, \tau)} \leq C\epsilon(1 - \tau)^{-\epsilon}\; \; \text{and}\;\; |\nabla u(\xi, \tau)| \leq C\epsilon, \quad \forall |\xi| < \frac{1}{4}, \; \forall \tau \in [0, 1).
\end{equation}
Repeat this step one again with $r = \frac{1}{4}$ by using \eqref{est:Newu12g12} and \eqref{equ:wtau}, we end up with
$$e^{u(\xi, \tau)} + |\nabla u(\xi, \tau)| \leq C\epsilon, \quad \forall |\xi| < \frac{1}{8}, \; \forall \tau \in [0, 1).$$
This concludes the proof of Proposition \ref{prop:noblowup}.
\end{proof}

\bibliographystyle{alpha}
%\bibliography{/Volumes/Data/Work/mybib} %for Mac
%\bibliography{D:/Works/mybib}

\begin{thebibliography}{STW96}

\bibitem[BB92]{BBsima92}
J.~Bebernes and S.~Bricher.
\newblock Final time blowup profiles for semilinear parabolic equations via
  center manifold theory.
\newblock {\em SIAM J. Math. Anal.}, 23(4):852--869, 1992.

\bibitem[BE89]{BEbook89}
J.~Bebernes and D.~Eberly.
\newblock {\em Mathematical problems from combustion theory}, volume~83 of {\em
  Applied Mathematical Sciences}.
\newblock Springer-Verlag, New York, 1989.

\bibitem[BK94]{BKnon94}
J.~Bricmont and A.~Kupiainen.
\newblock Universality in blow-up for nonlinear heat equations.
\newblock {\em Nonlinearity}, 7(2):539--575, 1994.

\bibitem[Bre90]{Breiumj90}
A.~Bressan.
\newblock On the asymptotic shape of blow-up.
\newblock {\em Indiana Univ. Math. J.}, 39(4):947--960, 1990.

\bibitem[Bre92]{Brejde92}
A.~Bressan.
\newblock Stable blow-up patterns.
\newblock {\em J. Differential Equations}, 98(1):57--75, 1992.

\bibitem[CFQ03]{CFQdcds03}
M.~Chleb{\'{\i}}k, M.~Fila, and P.~Quittner.
\newblock Blow-up of positive solutions of a semilinear parabolic equation with
  a gradient term.
\newblock {\em Dyn. Contin. Discrete Impuls. Syst. Ser. A Math. Anal.},
  10(4):525--537, 2003.

\bibitem[CW89]{CWjma89}
M.~Chipot and F.~B. Weissler.
\newblock Some blowup results for a nonlinear parabolic equation with a
  gradient term.
\newblock {\em SIAM J. Math. Anal.}, 20(4):886--907, 1989.

\bibitem[CZ13]{CZcpam13}
R.~C{\^o}te and H.~Zaag.
\newblock Construction of a multisoliton blowup solution to the semilinear wave
  equation in one space dimension.
\newblock {\em Comm. Pure Appl. Math.}, 66(10):1541--1581, 2013.

\bibitem[Dol85]{Doldqjmam85}
J.~W. Dold.
\newblock Analysis of the early stage of thermal runaway.
\newblock {\em The Quarterly Journal of Mechanics and Applied Mathematics},
  38(3):361--387, 1985.

\bibitem[Dol89]{Doldsjam89}
J.~W. Dold.
\newblock Analysis of thermal runaway in the ignition process.
\newblock {\em SIAM J. Appl. Math.}, 49(2):459--480, 1989.

\bibitem[EZ11]{EZsema11}
M.~A. Ebde and H.~Zaag.
\newblock Construction and stability of a blow up solution for a nonlinear heat
  equation with a gradient term.
\newblock {\em S$\vec{\rm e}$MA J.}, (55):5--21, 2011.

\bibitem[Fil91]{FILpams91}
M.~Fila.
\newblock Remarks on blow up for a nonlinear parabolic equation with a gradient
  term.
\newblock {\em Proc. Amer. Math. Soc.}, 111(3):795--801, 1991.

\bibitem[FK92]{FKcpam92}
S.~Filippas and R.~V. Kohn.
\newblock Refined asymptotics for the blowup of {$u_t-\Delta u=u^p$}.
\newblock {\em Comm. Pure Appl. Math.}, 45(7):821--869, 1992.

\bibitem[FM85]{FM85iumj}
A.~Friedman and B.~McLeod.
\newblock Blow-up of positive solutions of semilinear heat equations.
\newblock {\em Indiana Univ. Math. J.}, 34(2):425--447, 1985.

\bibitem[FP08]{FPtmj08}
M.~Fila and A.~Pulkkinen.
\newblock Nonconstant selfsimilar blow-up profile for the exponential
  reaction-diffusion equation.
\newblock {\em Tohoku Math. J. (2)}, 60(3):303--328, 2008.

\bibitem[GK89]{GKcpam89}
Y.~Giga and R.~V. Kohn.
\newblock Nondegeneracy of blowup for semilinear heat equations.
\newblock {\em Comm. Pure Appl. Math.}, 42(6):845--884, 1989.

\bibitem[GNZ16]{GNZpre16c}
T.~Ghoul, V.~T. Nguyen, and H.~Zaag.
\newblock Construction and stability of blowup solutions for a non-variational
  parabolic system.
\newblock {\em arXiv:1610.09883}, 2016.

\bibitem[GV93]{GVsiam93}
V.~A. Galaktionov and J.~L. V{\'a}zquez.
\newblock Regional blow up in a semilinear heat equation with convergence to a
  {H}amilton-{J}acobi equation.
\newblock {\em SIAM J. Math. Anal.}, 24(5):1254--1276, 1993.

\bibitem[GV96]{GVjde96}
V.~A. Galaktionov and J.~L. V\'azquez.
\newblock Blow-up for quasilinear heat equations described by means of
  nonlinear {H}amilton-{J}acobi equations.
\newblock {\em J. Differential Equations}, 127(1):1--40, 1996.

\bibitem[HV92]{HVasnsp92}
M.~A. Herrero and J.~J.~L. Vel{\'a}zquez.
\newblock Generic behaviour of one-dimensional blow up patterns.
\newblock {\em Ann. Scuola Norm. Sup. Pisa Cl. Sci. (4)}, 19(3):381--450, 1992.

\bibitem[HV93]{HVaihn93}
M.~A. Herrero and J.~J.~L. Vel{\'a}zquez.
\newblock Blow-up behaviour of one-dimensional semilinear parabolic equations.
\newblock {\em Ann. Inst. H. Poincar{\'e} Anal. Non Lin{\'e}aire},
  10(2):131--189, 1993.

\bibitem[Mer92]{Mercpam92}
F.~Merle.
\newblock Solution of a nonlinear heat equation with arbitrarily given blow-up
  points.
\newblock {\em Comm. Pure Appl. Math.}, 45(3):263--300, 1992.

\bibitem[MZ97a]{MZnon97}
F.~Merle and H.~Zaag.
\newblock Reconnection of vortex with the boundary and finite time quenching.
\newblock {\em Nonlinearity}, 10(6):1497--1550, 1997.

\bibitem[MZ97b]{MZdm97}
F.~Merle and H.~Zaag.
\newblock Stability of the blow-up profile for equations of the type
  {$u_t=\Delta u+\vert u\vert ^{p-1}u$}.
\newblock {\em Duke Math. J.}, 86(1):143--195, 1997.

\bibitem[MZ08]{MZjfa08}
N.~Masmoudi and H.~Zaag.
\newblock Blow-up profile for the complex {G}inzburg-{L}andau equation.
\newblock {\em J. Funct. Anal.}, 255(7):1613--1666, 2008.

\bibitem[NZ15]{NZcpde15}
N.~Nouaili and H.~Zaag.
\newblock Profile for a simultaneously blowing up solution to a complex valued
  semilinear heat equation.
\newblock {\em Comm. Partial Differential Equations}, 40(7):1197--1217, 2015.

\bibitem[NZ16a]{NZsns16}
V.~T. Nguyen and H.~Zaag.
\newblock Construction of a stable blow-up solution for a class of strongly
  perturbed semilinear heat equations.
\newblock {\em Ann. Scuola Norm. Sup. Pisa Cl. Sci., to appear}, 2016.

\bibitem[NZ16b]{NZens16}
V.~T. Nguyen and H.~Zaag.
\newblock Finite degrees of freedom for the refined blow-up profile for a
  semilinear heat equation.
\newblock {\em Ann. Scient. {\'E}c. Norm. Sup. to appear}, 2016.

\bibitem[Pul11]{PULmmas11}
A.~Pulkkinen.
\newblock Blow-up profiles of solutions for the exponential reaction-diffusion
  equation.
\newblock {\em Math. Methods Appl. Sci.}, 34(16):2011--2030, 2011.

\bibitem[Sou01]{SOUjde01}
P.~Souplet.
\newblock Recent results and open problems on parabolic equations with gradient
  nonlinearities.
\newblock {\em Electron. J. Differential Equations}, pages No. 10, 19 pp.
  (electronic), 2001.

\bibitem[ST01]{STcm01}
P.~Souplet and S.~Tayachi.
\newblock Blowup rates for nonlinear heat equations with gradient terms and for
  parabolic inequalities.
\newblock {\em Colloq. Math.}, 88(1):135--154, 2001.

\bibitem[ST07]{SThmj07}
S.~Snoussi and S.~Tayachi.
\newblock Large time behavior of solutions for parabolic equations with
  nonlinear gradient terms.
\newblock {\em Hokkaido Math. J.}, 36(2):311--344, 2007.

\bibitem[STW96]{STWiumj96}
P.~Souplet, S.~Tayachi, and F.~B. Weissler.
\newblock Exact self-similar blow-up of solutions of a semilinear parabolic
  equation with a nonlinear gradient term.
\newblock {\em Indiana Univ. Math. J.}, 45(3):655--682, 1996.

\bibitem[TZ16]{TZpre15}
S.~Tayachi and H.~Zaag.
\newblock Existence of a stable blow-up profile for the nonlinear heat equation
  with a critical power nonlinear gradient term.
\newblock {\em arXiv:1506.08306}, 2016.

\bibitem[Vel92]{VELcpde92}
J.~J.~L. Vel{\'a}zquez.
\newblock Higher-dimensional blow up for semilinear parabolic equations.
\newblock {\em Comm. Partial Differential Equations}, 17(9-10):1567--1596,
  1992.

\bibitem[Vel93]{VELtams93}
J.~J.~L. Vel{\'a}zquez.
\newblock Classification of singularities for blowing up solutions in higher
  dimensions.
\newblock {\em Trans. Amer. Math. Soc.}, 338(1):441--464, 1993.

\bibitem[VGH91]{GHVsovi91}
J.~J.~L. Vel{\'a}zquez, V.~A. Galaktionov, and M.~A. Herrero.
\newblock The space structure near a blow-up point for semilinear heat
  equations: a formal approach.
\newblock {\em Zh. Vychisl. Mat. i Mat. Fiz.}, 31(3):399--411, 1991.

\bibitem[Zaa98]{ZAAihn98}
H.~Zaag.
\newblock Blow-up results for vector-valued nonlinear heat equations with no
  gradient structure.
\newblock {\em Ann. Inst. H. Poincar{\'e} Anal. Non Lin{\'e}aire},
  15(5):581--622, 1998.

\bibitem[Zaa02]{ZAAaihp02}
H.~Zaag.
\newblock On the regularity of the blow-up set for semilinear heat equations.
\newblock {\em Ann. Inst. H. Poincar\'e Anal. Non Lin\'eaire}, 19(5):505--542,
  2002.

\end{thebibliography}

\def\cprime{$'$}

\vspace*{1cm}
\noindent $\begin{array}{ll}
\textbf{Tej-Eddine Ghoul},&\text{New York University in Abu Dhabi, ​Departement of Mathematics,}\\
\textbf{Van Tien Nguyen:} & \text{Computational Research Building A2,}\\
& \text{Saadiyat Island, P.O. Box 129188, Abu Dhabi, UAE.}\\
& \textit{Email: Teg6@nyu.edu}\\
& \textit{Email: Tien.Nguyen@nyu.edu}\\
&\quad \\
\textbf{Hatem Zaag:}& \text{Universit\'e Paris 13, Institut Galil\'ee, LAGA,}\\
& \text{99 Avenue Jean-Baptiste Cl\'ement,}\\
& \text{93430 Villetaneuse, France.}\\
&\textit{Email: Hatem.Zaag@univ-paris13.fr} 
\end{array}$

\end{document}